\documentclass[11pt]{article}
\usepackage{geometry}                
\usepackage{authblk}
\usepackage{graphicx}
\usepackage{amssymb}
\usepackage{amsthm}
\usepackage{amsmath}
\usepackage{mathabx}
\usepackage{enumitem}
\usepackage{silence}
\usepackage[square,numbers]{natbib}
\usepackage{xcolor}
\usepackage{algorithm}
\usepackage[noend]{algpseudocode}
\usepackage[section]{placeins}

\usepackage{titlesec}

\titleformat{\subsection}{\normalfont\normalsize}{\thesubsection}{1em}{}
\usepackage{hyperref}
\usepackage{nameref}
\usepackage{cleveref} 
\usepackage{url}

\newlength{\maxwidth}
\newcommand{\algalign}[2]
{\makebox[\maxwidth][r]{$#1{}$}${}#2$}

\newtheorem{remark}{Remark}[section]
\crefname{remark}{Remark}{Remarks}

\newtheorem{assump}{Assumption}[section]
\crefname{assump}{Assumption}{Assumptions}

\newtheorem{definition}{Definition}[section]
\crefname{definition}{Definition}{Definitions}

\newtheorem{theorem}{Theorem}[section]
\crefname{theorem}{Theorem}{Theorems}

\crefname{claim}{Claim}{claims}

\newtheorem{prop}{{Proposition}}[section]
\crefname{prop}{Proposition}{Propositions}

\newtheorem{lemma}{{Lemma}}[section]
\crefname{lemma}{Lemma}{Lemma}

\newtheorem{cor}{{Corollary}}[section]
\crefname{cor}{Corollary}{Corollaries}
\Crefname{cor}{Corollary}{Corollaries}

\crefname{algorithm}{algorithm}{algorithms}
\Crefname{algorithm}{Algorithm}{Algorithms}

\crefname{section}{Appendix}{Appendices}

\definecolor{darkred}{rgb}{.7,0,0}
\definecolor{darkblue}{rgb}{0,0,.7}
\definecolor{darkgreen}{rgb}{0,.7,0}
\definecolor{darkbrown}{rgb}{0.8,0.4,0.4}

\title{Early stopping for Ensemble Kalman-Bucy Inversion}
\author[]{Maia Tienstra}
\author[]{Sebastian Reich}
\affil[]{Institut für Mathematik, Universität Potsdam}

\begin{document}
\maketitle

\begin{abstract}
    Bayesian linear inverse problems aim to recover an unknown signal from noisy observations, incorporating prior knowledge. This paper analyses a data-dependent method to choose the scale parameter of a Gaussian prior. The method we study arises from early stopping methods, which have been successfully applied to a range of problems, such as statistical inverse problems, in the frequentist setting. These results are extended to the Bayesian setting. We study the use of a discrepancy-based stopping rule in the setting of random noise, which allows for adaptation. Our proposed stopping rule results in optimal rates for the reparameterized problem under certain conditions on the prior covariance operator. We furthermore derive for which class of signals this method is adaptive. It is also shown that the associated posterior contracts at the same rate as the MAP estimator and provides a conservative measure of uncertainty. We implement the proposed stopping rule using the continuous-time ensemble Kalman--Bucy filter (EnKBF). The fictitious time parameter replaces the scale parameter, and the ensemble size is appropriately adjusted in order not to lose the statistical optimality of the computed estimator. With this Monte Carlo algorithm, we extend our results numerically to a nonlinear problem. 
\end{abstract}

%
\section{Introduction}

Bayesian inference methods are widely used in statistical inverse problems. A major challenge 
is the selection and computational implementation of suitable prior distributions. This problem can be addressed by 
using hierarchical Bayesian methods and Bayesian model selection. More recently, however, a frequentist analysis of Bayesian methods has gained popularity \cite{Stuart}, where choosing the prior is synonymous with choosing the amount of regularisation. Furthermore, a frequentist analysis of posterior credible intervals has also become an active area of research \cite{knapik2011, knapik2016, GineNickl2016book}. In this paper, we follow both lines of research and analyse an adaptive choice of the prior. We build on recent results on adaptive choice of the regularisation parameter for statistical inverse problems, which have been studied extensively in \cite{BlanchardHoffmannReiss, BlanchardHoffmannReiss-bis, Stankewitz}. In these papers, the regularisation parameter is chosen using statistical early stopping. Early stopping is a method that measures some notion of error and stops an iterative process via a defined stopping rule. We extend these methods as an empirical Bayesian method for selecting the scale parameter of the prior covariance. In addition, the ensemble Kalman Filter (EnKF) and its continuous-time formulation of the ensemble Kalman-Bucy Filter (EnKBF) have become popular methods for performing Bayesian inference on high-dimensional inverse problems. See \cite{CRS22} for an overview of EnKF and the closely related ensemble Kalman inversion (EKI) \cite{iglesias2013ensemble}. The convergence rates of adaptive EKI have previously been studied for deterministic linear inverse problems in \cite{Parzer}. Here we combine this work with the Bayesian frequentist perspective \cite{knapik2011, knapik2016, GineNickl2016book}, statistical early stopping \cite{BlanchardHoffmannReiss, BlanchardHoffmannReiss-bis, Stankewitz}, and continuous-time EnKBF implementations \cite{reich10}.

%
\subsection{Problem Formulation}\label{subsec:linear_diag}
%
We will now recall the Bayesian inverse problem setting of \cite{knapik2011, Szabo}. We are interested in recovering the ground truth signal $\theta^\dagger$ from the following observations $Y$, which we believe to be generated by the following model
\begin{equation} \label{eq:model}
Y = G\theta^\dagger  + \delta \Xi,
\end{equation}
where $\delta = \frac{\nu}{\sqrt{n}} > 0 $, and $\nu$ is assumed to be unknown. Here $G: H_1 \to H_2$  denotes a known linear, compact, continuous operator between two Hilbert spaces  $H_1$ and $H_2$ with inner products $\langle \cdot ,\cdot \rangle_1$ and $\langle \cdot ,\cdot \rangle_2$, respectively. We will later need to project to a finite-dimensional subspace of $H_{1,2}$, and the dimension of $H_{1,2}$ depends on $n$. The norms of $H_1$ and $H_2$ are denoted by $|| \cdot ||_1$ and $||\cdot ||_2$ respectively. We will denote the adjoint of an operator $A$ between two Hilbert spaces by $A^\ast$.

The measurement error $\Xi$ is assumed to be Gaussian white noise and $\delta$ denotes the noise level which we will study in the limit $n \rightarrow \infty$. The noise $\Xi$ is not an element of $H_2$, but we can define it as a Gaussian process $(\Xi_h : h \in H_2)$ with mean $0$, and covariance $\rm {cov}(\Xi_h , \Xi_{h^\prime}) = \langle h , h^\prime \rangle_2$. The observations are then driven by this process. Thus, we observe a Gaussian process $Y=(Y_h : h \in H_2)$ with mean and covariance given by
\begin{equation}
    \mathbb{E}Y_h = \langle  G \theta^\dagger , h\rangle_2, \quad \rm {cov}(Y_h, Y_{h^\prime}) = \frac{1}{n} \langle h, h^\prime \rangle_2.
\end{equation}

In this paper, we follow a Bayesian perspective and
place a Gaussian prior over the unknown parameter $\theta$, which is conjugate to \eqref{eq:model}, implying that the posterior will be Gaussian and analytic. We also further assume that the true $\nu=1$, to simplify the notation. Particularly, we consider a family of Gaussian priors $\mathcal{N}(0,\tau_{n}^2 C_0)$ with covariance operator $C_0:H_1\to H_1$ and where $\tau_{n} >0$ is the scaling parameter of interest. 

\begin{remark}
    We remark that in this paper, we will use the scale parameter to be $\tau_n^2$ instead of $n \tau_n^2$ as was done in \cite{knapik2011}. We did this to be in line with the setting of \cite{EngHanNeu96} and \cite{BlanchardHoffmannReiss}, of which our proofs more closely follow. 
\end{remark}
\begin{prop} (Prop. 3.1 in \cite{knapik2016})
\label{prop:bayesian_setup}
    For given $\tau_{n}>0$, the prior distribution for $\theta$ is $\mathcal{N}(0, \tau_{n}^2 C_0)$ and $Y$ given $\theta$ is $\mathcal{N}(G \theta, n^{-1} I)$ distributed. Then the conditional distribution of $\theta$ given $Y$, the posterior, is Gaussian $\mathcal{N}(\widehat{\theta}_{\tau_{n}}, C_{\tau_{n}})$ on $H_1$ with mean
    \begin{equation} \label{eq:Bayes_mean}
        \widehat{\theta}_{\tau_{n}} := K_{\tau_{n}} Y
    \end{equation}
    and covariance operator
    \begin{equation} \label{eq:Bayes_covariance}
        C_{\tau_{n}} := \tau_{n}^2 C_0 - \tau_{n}^2 K_{\tau_{n}} \left( G C_0 G^\ast + \frac{1}{ \tau_{n}^2}I \right)K_{\tau_{n}}^\ast ,
    \end{equation}
    where the Kalman gain $K_{\tau_{n}}: H_2 \rightarrow H_1$ is the linear continuous operator given by
    \begin{equation}
        K_{\tau_{n}} := C_0 G^\ast \left( 
        G C_0 G^\ast + \frac{1}{ \tau_{n}^2} I \right)^{-1}.
    \end{equation}
\end{prop}

We remark that a rigorous construction of the Bayesian set-up for infinite-dimensional Hilbert space (which will be the limiting object when we let the dimension of our subspaces go to  $\infty$) can be found in \cite{Ghosal, GineNickl2016book}. We recall that the mean
$\widehat{\theta}_{\tau_{n}}$ also arises formally as the minimiser of the Tikhonov functional,
\begin{equation}
 \label{eq:loss_functional}
     \mathcal{L}(\theta) =   \|P_n(G\theta - Y)\|^2_2 + 
\tau_{n}^{-2} ||P_n(C_0^{-1/2} \theta|)|^2_1 .
\end{equation}
where $P_n$ is an appropriate projection operator onto a finite-dimensional subspace. For each $n$, we will define a projection $P_n$ such that $H_{1,2}$ have dimension $D(n)$. In the finite setting, we drop specification of the norm as all norms are equivalent up to a constant that does not depend on $n$. Moreover, we see that the scale parameter $\tau_{n}$ of the prior becomes the regularisation parameter and that the estimator crucially depends on the choice of $\tau_{n}$. This connection between the Bayesian inverse problem and Tikhonov regularisation has been extensively studied, see for example \cite{Stuart}. The question we wish to answer is: can we choose, $\tau_{n}$ depending on $Y$, such that $\widehat{\theta}_{\tau_{n}}$ provides an adaptively optimal frequentist estimator for $\theta^\dagger$ and $C_{\tau_{n}}$ corresponds to the frequentist uncertainty in the asymptotic limit $n \rightarrow \infty$. This paper answers this question.
To  choose the $\tau_n$, we will use early stopping, which is defined as follows: Suppose that for some given iterative method and for each $\tau_n \in \mathbb{R}_{+} \cup \{0\}$ we have a sequence of estimators 
\begin{align*}
    (\widehat{\theta}_{\tau_n})_{\tau_n}
\end{align*}
such that they minimise,
\begin{align*}
    \widehat{\theta}_{\tau_n} = \text{argmin } \mathcal{L}_{(\tau_n)}(\theta)
\end{align*}
and are ordered in decreasing bias and increasing variance. The goal of early stopping is to stop this iterative method exactly when the bias and variance of the estimator are balanced. An estimator for the asymptotic bias is given by the residuals, something that we can always measure. The residuals at time $\tau_n$ are given as
\begin{align} \label{eq:discrepany_proj}
    R_{\tau_{n}} := || P_n(Y - G \widehat{\theta}_{\tau_{n}})||^2,
\end{align}
To stop the iterative process, we must define a stopping rule
\begin{align}
\label{eq:dp_stopping_time}
    \tau_{\rm dp}(n) := \inf\,\{\tau_{n} > 0 : R_{\tau_{n}} \leq \kappa \}
\end{align}
for some threshold $\kappa > 0$ that is also a function of the data and is thus depending on $n$. 
\begin{remark}
    We will further drop the dependency on $n$ to simplify the notation, but we remark that $\tau_{\rm dp}$ is an estimator that depends on the data $Y$ and is a function of $n$. 
\end{remark}
Formulation (\ref{eq:dp_stopping_time}) is often referred to as the discrepancy principle \cite{EngHanNeu96, Mathe} in the case that the noise level is known and $\kappa$ is chosen proportional to this noise level. Other choices for $\kappa$ have been studied, see Subsection \ref{subsection: previous_work}. 

\begin{definition}
    An estimator $\widehat{\theta}$ of $\theta$ is called minimax for a certain risk function $R(\theta, \widehat{\theta})$ if $\widehat{\theta}$ achieves the smallest maximum risk among all estimators. That is 
    \begin{equation*}
        \underset{\theta \in \Theta}{\sup}R(\theta, \widehat{\theta}) = \underset{\widehat{\theta}}{\rm inf}  \underset{\theta \in \Theta}{\ \sup} R(\theta, \widehat{\theta}).
    \end{equation*}
\end{definition}

\begin{definition}(See Chapter 6  in \cite{Johnstone})
    A scale parameter $\tau_n$ and its associated estimator $\widehat{\theta}_{\tau_n}$ are called optimal if it achieve the minimax rate as $n \rightarrow \infty$ for a given Sobolev regularity of the unknown $\theta^\dagger$. We call the estimator adaptive if it does not require knowledge of the Sobolev regularity. A method is then adaptively optimal if it produces an estimator $\widehat{\theta}_{\tau_n}$ that is both adaptive and optimal. 
\end{definition}

\noindent
The challenge then is how to choose $P_n$ and $\kappa$ such that stopping according to (\ref{eq:dp_stopping_time}) is adaptively optimal. In the following section, we summarise some of the previous contributions to solving this problem from a statistical perspective.

%
\subsection{Previous work on early stopping and ensemble Kalman inversion}
\label{subsection: previous_work}
%

The works of Blanchard, Hoffman, and Reiß (\cite{BlanchardHoffmannReiss-bis}) and that of Blanchard and Mathé (\cite{Mathe}) study the discrepancy principle (\ref{eq:dp_stopping_time}) in a frequentist setting. More precisely, in \cite{Mathe}, they consider a variety of regularisation methods, such as Tikhonov, and consider the weighted residuals $|| (G^\ast G)^{1/2}(Y - G\widehat{\theta}_{\tau_{n}}) ||_2^2$ instead of the projection operator $P_n$ in (\ref{eq:discrepany_proj}). In the case that the true Sobolev regularity is known, this method can achieve optimal rates. However, no adaptive generalisation is currently available. In \cite{BlanchardHoffmannReiss}, they consider the discretised version of \eqref{eq:loss_functional}, with identity covariance matrix. The discretisation is dependent on the estimation of the effective dimension, and is an integral part of their analogue of (\ref{eq:dp_stopping_time}). In this regime, the method is adaptive for a small range of signals. In \cite{Szabo}, the Bayesian perspective is studied, and they use an empirical Bayesian method which maximises the log-likelihood for $\tau_{n}$. In this setting, they are able to achieve optimal rates as long as the prior is smooth enough. 
In this work, we instead extend the discrepancy principle (\ref{eq:dp_stopping_time}) to Bayesian estimators of the form (\ref{eq:Bayes_mean}).

The ensemble Kalman filter (EnKF) has become a popular derivative-free method for approximating posterior distributions. The so-called time-continuous formulations, ensemble Kalman-Bucy filter (EnKBF), have been first proposed in \cite{br10a,reich10,br11}. These formulations become exact again for linear Gaussian problems. A frequentist perspective on the EnKBF has already been explored in \cite{RR20,R22}. 

The EnKF has also been utilised as a derivative-free optimisation method. These variants of the EnKF are often collected under the notion of ensemble Kalman inversion (EKI). EKI can be used for finding the minimiser of the cost functional \eqref{eq:loss_functional} see \cite{iglesias2013ensemble,chada2020tikhonov} and the review paper \cite{CRS22}. Algorithms for selecting the Tikhonov regularisation parameter within EKI have been discussed, for example, in \cite{weissmann2022adaptive}. Similarly, a discrepancy principle-based stopping rule has been implemented for EKI in \cite{iglesias2016regularizing,iglesias2020adaptive}. A similar perspective of viewing the regularisation parameter as a time parameter can be found in \cite{Parzer}. Here, they analysed stochastic estimation methods for the ensemble covariance operator and adaptively increased the ensemble size with each time step. In this paper, we will take a Bayesian perspective and keep the ensemble size fixed and adapt to the smoothness of the ground truth parameter. Finally, we remark that the results of this paper build on those found in \cite{BlanchardHoffmannReiss}, and can be seen as an extension of their results.    

\subsection{Main contributions and paper outline}

In this paper, we study the discrepancy principle-based stopping rule considered in \cite{BlanchardHoffmannReiss} as an empirical method for choosing the scale parameter of the prior covariance. By considering the Bayesian setting of Tikhonov regularisation \cite{knapik2011}, we can extend the setting of \cite{BlanchardHoffmannReiss}, and then provide a Bayesian analysis of the asymptotic behaviour of the posterior stopped at time $\tau_{\rm dp}$. To do this reparameterize the problem of estimating $\theta$ into estimating $\widetilde{\theta}$. Furthermore, we derive for which $\ell^2(\mathbb{N})$ the sequence spaces of Sobolev regular functions; this method is optimal for $\widetilde{\theta}$ and provides an adaptation interval for this method. We then reformulate the regularisation parameter as a time parameter, where we sequentially approach the final posterior by studying the time-continuous ensemble Kalman-Bucy filter. In the linear setting, the associated McKean-Vlasov evolution equations for mean and covariance are exact and provide an iterative process to transform the prior distribution. 

This paper is structured as follows: In \cref{sec:theory} we introduce the mathematical assumptions of our model. We provide the details of the projection matrix $P_n$ and the stopping rule \eqref{eq:dp_stopping_time}. We then reformulate \eqref{eq:loss_functional} and derive the associated filter function. Using the filter function, we extend the results of \cite{BlanchardHoffmannReiss} and show that we can achieve minimax estimation rates for certain classes of signals. We then show that the stopped posterior contracts optimally for $\widetilde{\theta}$ and provides a conservative measure of uncertainty. In \cref{sec:EKI} we formulate the Bayesian inverse problem of \eqref{eq:model} in terms of the time-continuous ensemble Kalman-Bucy filter and introduce the associated McKean-Vlasov evolution equations. In this section, we formulate the regularisation parameter as a time parameter. We then show that taking a finite particle, a finite-dimensional approximation of the posteriors, leads to an error that is smaller than the statistical minimax error.  In \cref{sec:numerical_implementation}, we provide the discrete-time approximation of the process of the continuous-time EnKBF and give the associated algorithm. We also extend the results of this paper numerically to a non-linear inverse problem in \cref{sec:non-linear}. All numerical results can then be found in  \cref{subsec:linear} and \cref{subsec:nonlinear_numerics}. Finally, the conclusions can be found in  \cref{sec:conclusion}.  

\subsection{Further notation} We define the following additional notation. For two numbers $a$ and $b$, we denote the minimum of $a$ and $b$ by $a \wedge b$. For two sequences $(a_n)_n$ and $(b_n)_n$ in $\mathbb{R}_{+}$, $a_n \lesssim b_n$, respectively $a_n \gtrsim b_n$ denote inequalities up to a multiplicative constant. $a_n \asymp b_n$ denotes that $a_n \lesssim b_n$ and $a_n \gtrsim b_n$ hold. $\ell^2(\mathbb{N})$ denotes the space of sequences that are square summable with index $i \in \mathbb{N}$, and its norm is denoted by $\|\cdot \|_{\ell^2(\mathbb{N})} = \left(\sum_i a_i^2\right)^{1/2}$. For a random variable, $Y$, denote the distribution of $Y$ by $\mathbb{P}_Y$. The space of bounded linear operators mapping from $H_1$ to $H_2$ is denoted by $\mathcal{L}(H_1,H_2)$ , with the respective norm denoted by $\| \cdot \|_{\mathcal{L}(H_1,H_2)}$. For $T$ a trace class operator with singular values, $(a_i)_{i \in \mathbb{N}}$ the trace norm (or Hilbert-Schmidt norm) is $\|T\|_{\mathbb{T}} = \text{Tr}(TT^\ast)^{1/2} = \sum_{i=1}^\infty a_i$ and its adjoint by $T^\ast$. When we are in the projected finite space, we will then denote the adjoint by $T^T$, that is $T^T := P_n(T^*)$. We can then view the class of trace class operators as sequences in $\ell^1$ via their associated sequences of singular values. Finally, for $L$ an operator between Hilbert spaces, we denote the domain of $L$ by $\mathcal{D}(L)$. 
%

%
\section{Theoretical results on adaptive estimation} \label{sec:theory}
%
As $G$ is a linear compact operator by the Spectral Theorem, the eigenfunctions $(v_i)_{i\in \mathbb{N}}$, of $G^\ast G$ form an orthonormal basis of $H_1$. Denote the bounded eigenvalues of $(G^\ast G)^{1/2}$ by, 
\begin{align}
    \sigma_1 \geq \sigma_2 \geq ... > 0.
\end{align}
The following sequence space model is equivalent to observing  \eqref{eq:model}, see \cite{Johnstone}, and is written as 
\begin{align}
    \label{eq:sequence_space_model}
    Y_i = \sigma_i \theta_i^\dagger + \frac{1}{\sqrt{n}} \epsilon_i,
\end{align}
for $i\ge 1$, where $\theta_i^\dagger = \langle \theta^\dagger , v_i\rangle_1$  for $i \in \mathbb{N}$. Furthermore, all $\epsilon_i$ are i.i.d. $\mathcal{N}(0,1)$ with respect to the conjugate basis
$(u_i)_{i\in \mathbb{N}}$ of the range of $G$ in $H_2$ defined by 
\begin{equation}
    G v_i = \sigma_i u_i
\end{equation}
and $Y_i = \langle Y,u_i\rangle_2$.

We project this infinite-dimensional inverse problem to a finite-dimensional one of dimension $D(n)$. We will later see that defining $D(n)$ is necessary to define the stopping rule based on the discrepancy principle. See, for example \cite{BlanchardHoffmannReiss, BlanchardHoffmannReiss-bis, Stankewitz}. The truncation is performed in sequence space (\ref{eq:sequence_space_model}) by truncating all coefficients, $i$, larger than an appropriate dimension $D(n)$. Specifically, we observe
\begin{equation}
\label{eq:projector}
\langle P_n Y,u_i \rangle_2 = \left\{ \begin{array}{ll}
Y_i & {\rm if}\,\,i\le D(n)\\
0 & {\rm otherwise}
\end{array} \right.
\end{equation}
The truncation is chosen to depend on the sample size, and we consider the limit  $n \rightarrow \infty$, where the noise goes to zero.  We then have, of course, that $D(n) \rightarrow \infty$. 
%
 \subsection{Structural Assumptions}
 \label{subsec:structural_ass}
%

 We assume that the inverse problem is polynomially ill-posed, where the degree of ill-posedness is given by some parameter $p >0$. That is, the eigenvalues decay as
\begin{equation}
    \sigma_i \asymp i^{-p}, \quad i=1,...,D(n), \quad \forall n \in \mathbb{N} 
\end{equation}

We choose an entry-wise Gaussian prior over the coefficients of the variable of interest $\theta$ of the form, $\theta_i \overset{iid}{\sim} \mathcal{N}(0, \tau_{n}^2\lambda_i)$. For $\tau_{n} > 0$, we assume that the $\lambda_i's$ decay as
 \begin{equation}
 \label{ass:decay_prior}
     \lambda_i \asymp  i^{-2\alpha-1}, \quad i=1,...,D(n), \quad \forall n \in \mathbb{N} 
\end{equation} 
for $\alpha > 0 $.
 We furthermore assume that the observations are generated given some true signal that lies in the Sobolev space, $S^{\beta^\prime}$, where $\beta^\prime$ denotes the regularity of the signal. More specifically, 
 \begin{align}
     \theta^\dagger \in S^{\beta^\prime}:= \{\theta \in H_1 : \| \theta \|^2_{\beta^\prime} < \infty \}
 \end{align}
 where the norm is defined as  
 \begin{align}
     \theta=(\theta_i)_{i=1,...,D(n)} \mapsto \| \theta \|_{\beta^\prime}^2 := \sum_{i=1}^{\infty} i^{2 \beta^\prime} (\theta_i)^2.
 \end{align}
 We emphasise that the above two conditions hold for all $n \in \mathbb{N}$, and hence must also hold in the limit when $D(n) \rightarrow \infty$. Intuitively, these spaces consist of sequences of coefficients, $(\theta_i) \in \ell^2$, that decay faster to zero than the sequence $(i^{2\beta})$ for $i \in \mathbb{N}$. 

With these assumptions, the posterior is denoted by $\Pi_{n,\tau_{n}}$ to indicate the dependence on both $n$ and the scale parameter $\tau_{n}$. The entry-wise posterior is given as  
\begin{equation}
\theta_i \mid Y_i \sim \mathcal{N} \left(\frac{ \tau_{n}^2 \lambda_i \sigma_i Y_i}{1 +  \tau_{n}^2 \lambda_i \sigma^2_i}, \frac{\tau_{n}^2 \lambda_i}{1 +  \tau_{n}^2 \lambda_i \sigma_i}\right).
\end{equation}
The estimator (\ref{eq:Bayes_mean}) is then given in sequence space as 
\begin{align} \label{eq:estimator component wise}
    \widehat{\theta}_{i,\tau_{n} } = \frac{ \tau_{n}^2 \lambda_i \sigma_i Y_i}{1 + \tau_{n}^2 \lambda_i \sigma^2_i}.
\end{align}

Our signal in the projected space will lie in $r$ radius balls  of $S^{\beta^\prime}$, defined as
 \begin{align}
     \theta^\dagger \in S^{\beta^\prime}(r,D(n)) := \{\theta \in \mathbb{R}^{D(n)}:\sum_{i=1}^{D(n)} i^{2 \beta^{\prime}} \theta_i^2 < r \}.
 \end{align}

\begin{remark}
    By the spectral theorem, as $G$ is linear compact and self-adjoint on $H_1$, there exists an orthonormal basis of $H_1$ consisting of the eigenfunctions of $G$. We furthermore have that for any orthonormal basis of $H_1$ there is an isometry between $H_1$ and $\ell^2(\mathbb{N})$ \cite{Conway1990}. Thus, we can define an isometry using the basis of $G^\ast G$, which is defined as 
    \begin{equation*}
        h \mapsto \{\langle h, v_n \rangle \}_{n \in \mathbb{Z}}
    \end{equation*}
    where $(v_i)_{i\in \mathbb{N}}$, are the eigenfunctions of $G^\ast G$. From this section onward, when the norm is unspecified, it should be understood to be the standard $\ell^2(\mathbb{N})$ norm given this isometry.
\end{remark}
We recall that the minimax rate of estimation over the unit ball in $S^{\beta^\prime}$ is of order  $n^{-{\beta^\prime}/ (2{\beta^\prime} + 2p +1)}$ \cite{Tsybakov}. The effective dimension of our observed signal is given by
 \begin{equation}
     d_{\rm eff} \asymp n^{1/(2{\beta^\prime} +2p +1)}.
 \end{equation}
\begin{remark}
    In our setting, the effective dimension is the level at which, if we truncate the signal $Y$ to $P_n Y$ with $D(n) = d_{\rm eff}$, then the approximation error we make by estimating those modes to be zero is less than the minimax rate.
\end{remark} 
  As we do not know $\beta$,  the true smoothness of our signal, we choose $D(n) \leq n$ to be at least as fine as 
 \begin{equation} \label{eq:trunctation_dimension}
     D(n)  \asymp n^{1  / (2p+1)}
 \end{equation}
and so it depends only on $p$, the smoothness of the forward operator, which we know. This is an upper bound for $d_{\rm eff}$, as $\beta \geq 0$. 
Our goal is then to recover the first $D(n)$ coefficients of the signal $\theta^\dagger$. From here onward, we set $D(n)$ to be as in \cref{eq:trunctation_dimension}.

In deterministic inverse problems, it is assumed that the noise level is known. Then it is possible to implement the discrepancy principle such that we stop at the first iteration when $R_{\tau_{n}} \lesssim D(n)/n $.

\begin{remark}
In \cite{BlanchardHoffmannReiss}, they consider the setting with no regularisation operator $C_0$, and the dimension of the approximation space is $D(n)$,  with unknown noise $\delta$ in (\ref{eq:model}).  They show that for,
\begin{equation}
\label{eq:kappa}
    \kappa \asymp D(n) \delta^2
\end{equation}
stopping according to (\ref{eq:dp_stopping_time}) is adaptively optimal. The theory holds for slight deviations of this choice of $\kappa$, mainly that $\kappa$ can be chosen such that $| \kappa  - D(n) \delta^2| \leq c_n \sqrt{D(n) \delta} $ \cite{BlanchardHoffmannReiss} due to estimation of $\nu$ in $\delta$. We emphasise that the stopping rule depends on the truncation dimension, $D(n)$, as also discussed in this paper.  In our setting, we want to allow for regularising with some operator $C_0$ which then connects to the Bayesian interpretation.  
\end{remark}

\subsection{Theoretical Results}
In this section, we will show that choosing the regularisation parameter from early stopping as defined in \ref{eq:dp_stopping_time} results in a minimax estimator in the frequentist setting. We will also show that stopping the evolution of the prior into the posterior for $\widetilde{\theta}$, using early stopping, results in the posterior that contracts optimally, as well as having conservative frequentist coverage. Finally, we will show that early stopping is optimally adaptive only in a specific setting, which results from a condition on the prior smoothness in relation to the true smoothness. We will now always work in the $D(n)$ subspaces and no longer write $P_n$ everywhere. The norm, if not specified, should be understood as the standard Euclidean norm for vectors. First, recall the minimisation problem in \eqref{eq:loss_functional} where we wished to minimise 
\begin{equation}
\label{eq:gen_tik}
     \mathcal{L}(\theta) =    \|G\theta - Y\|^2 + 
\tau_{n}^{-2} ||L \theta||^2.
\end{equation}
For ease of notation, we now denote the inverse square root of the covariance operator of the prior by $L$. That is, we set 
\begin{equation}
\label{def:prior_operator}
    L : = C_0^{-1/2}.
\end{equation} 
From the structural assumptions on $C_0^{-1/2}$ in \eqref{ass:decay_prior} we see that $C_0$ is positive semi-definite, therefore both $C_0^{-1}$ and $C_0^{-1/2}$ exist and have eigenvalues denoted by 
\begin{equation}
    \lambda_{i}^{-1} \asymp i^{1+2\alpha}
\end{equation}
and
\begin{equation}
\label{ass:decay_L}
    \lambda^{-1/2} \asymp i^{1/2 + \alpha}.
\end{equation}
For $\theta \in \mathcal{D}(L)$, 
\begin{equation}
    \label{def:tilde_theta}
    \tilde \theta := L \theta. 
\end{equation}
Assume now that $\widetilde{\theta}^\dagger \in S^\beta$. We then rewrite the minimisation problem as 
\begin{equation}
    \label{eq:shifted_minproblem}
    \mathcal{L}(\tilde \theta) = \|A \tilde \theta  - Y\|^2 +  \tau_{n}^{-2}\|\tilde \theta\|^2.
\end{equation}
The forward operator $A=GL^{-1}$ is the new forward operator where $A: P_n H_{\tilde 1} \rightarrow P_n H_{2}$. 
\begin{remark}
    \label{lem:G_commutes}
       $G$ commutes with $L^{-1}$, and $L^{-1}$ is injective.  The first statement is satisfied as $G$ and $C_0^{-1/2}$ are diagonalisable with respect to the basis of $G$. Moreover, $L^{-1}$ is a linear compact operator, and  $C_0^{-1/2}$ is the inverse linear compact operator. So then, as the inverse exists because $C_0^{-1/2}$ is positive semi-definite.  
\end{remark}
Minimising the reparameterized problem \eqref{eq:shifted_minproblem}, as well as a priori and a posteriori rules for finding $\tau_n$, has been studied in \cite{EngHanNeu96, Rastogi} in the deterministic bounded noise setting. For a Bayesian perspective of this problem, see \cite{Gugushvili_2020}. Our analysis follows the techniques used in chapter 4 in \cite{EngHanNeu96} and that of \cite{BlanchardHoffmannReiss}. 
\begin{remark}
    The deterministic bounded noise setting and the statistical one are connected. In the deterministic setting, we assume that the data is perturbed and that we see observations $y^\delta$ such that 
    \begin{equation}
        || y^\delta - y|| \leq \delta
    \end{equation}
    where $\delta$ is called the noise level. The discrepancy principle then chooses the minimum $\tau_n$ such that
    \begin{equation}
    \label{eq:dp_deterministic}
        ||G\hat \theta_{\tau_n} - G \theta^\dagger|| \leq C \delta
    \end{equation}
    holds. 
    
    In the statistical setting, we assume we see random observations where the variance of $Y$ is scaled by $\delta = \frac{\nu}{\sqrt{n}}$. In this problem, we need to estimate $\nu$, and $n$ corresponds to the sample size. Further, we now assume that the variance of $Y$ projected into the $D(n)$ dimensional subspace of $H_2$, $\mathbb{E}\left[\|Y\|_2^2\right]$, is bounded by some order of $n$. We note that $Y \sim \mathcal{N}(0,\Sigma)$, then 
    \begin{equation}
        \mathbb{E}\left[\|Y\|_2^2\right] =\text{Var}(Y)= \Sigma
    \end{equation}
    From here we can see that the assumption that $\delta = \frac{\nu}{\sqrt{n}}$ is equivalent to 
    \begin{equation}
         ||\Sigma||_\mathbb{T} = tr(\Sigma) \leq n\delta^2 \leq \nu^2  
    \end{equation}
    We then see that \eqref{eq:kappa} is the normalised version of \eqref{eq:dp_deterministic}.
\end{remark}
The problem in the statistical setting is to show that bounding the residuals according to \eqref{eq:kappa} results in a bound for  $\mathbb{E}_Y \|\hat \theta_{\tau_{\rm dp}} - \theta^\dagger\|$, where $\tau_{\rm dp}$ is chosen according to some stopping rule. What we gain is that this method to choose $\tau_n$ is adaptive. To do this, we will need to define a filter function for the above minimisation problem. First, denote  the eigenvalues of $A^TA$ by $\tilde \sigma_{i}^2$, we from the assumptions on $G$ and $C_0$  that 
\begin{equation}
\label{ass:decay_A}
 \tilde \sigma_{i} \asymp i^{-(p+1/2+\alpha)}. 
\end{equation}
As $p, \alpha \geq 0$.  We see that we still have a polynomial decay in the eigenvalues.  
The minimiser of \eqref{eq:shifted_minproblem} is 
\begin{equation}
\label{minimiser_of_gentik}
    \tilde\theta_{\tau_{n}^{-2}} = (A^{T}A + \tau_{n}^{-2}I^{T}I)^{-1}A^{T}Y.
\end{equation}
Let us simplify the notation a bit and set $\tau_{n}^{-2}= \alpha$.
 \begin{lemma}
     Let $(\tilde\sigma_i ; \tilde v_i,\tilde u_i)$ be a singular system for $A$. Then the spectral filter for this generalised Tikhonov regularisation 
    \begin{equation}
           \|A \tilde \theta  - Y\|^2 +  \alpha \|\tilde \theta\|^2
    \end{equation}
    is given by
    \begin{equation}
    \label{eq:filter function}
        g_{\alpha}(\tilde \sigma_i^2) = \frac{1}{\tilde \sigma_{i}^2+ \alpha}. 
    \end{equation}
\end{lemma}
\begin{proof}
    Define
    \begin{equation}
          f_{\alpha}(\tilde \theta) = ||A \tilde \theta - Y||^{2}+ \alpha||\tilde \theta||^{2}
    \end{equation}
    Then $lim_{||\tilde \theta||\rightarrow \infty}f_{\alpha}(\tilde \theta) = + \infty$ and $f_\alpha$ is strictly convex, and $f_\alpha$ has a unique minimiser for each fixed $\alpha$. 
    \begin{equation}
        f^{\prime}(\tilde \theta)h = 0, \quad \forall h \in H_1.
    \end{equation}
    The derivative of $f_{\alpha}$ with respect to $\tilde \theta$ is, 
    \begin{equation}
        f^{\prime}(\tilde \theta)h = 2(A\tilde \theta -Y)^{T}A+ 2\alpha(  \tilde \theta)^T
    \end{equation}
    then 
    \begin{align}
    (A\tilde \theta -Y)^{T}A+  \alpha \tilde \theta^{T} &= 0  \nonumber  \\
    \iff \tilde \theta^{T}A^{T}A-Y^{T}A + \alpha \tilde \theta^{T} &=0  \nonumber \\
    \iff (A^{T}A+ \alpha I)\tilde \theta - Y^{T}A &=0  \nonumber \\
    \iff (A^{T}A+ \alpha I)\tilde \theta &=Y^{T}A \nonumber  \\
    \iff (A^{T}A+ \alpha I)\tilde \theta &=A^Ty
    \end{align}
    Then 
    \begin{equation}
        \tilde \theta_\alpha = (A^{T}A+ \alpha I)^{-1}A^{T}Y.
    \end{equation}
    Now define 
    \begin{equation}
        g_{\alpha}(A) =  (A^{T}A+ \alpha I)^{-1}.
    \end{equation}
\end{proof}
We reparametrize as in \cite{BlanchardHoffmannReiss}, where we will use this parameterisation to bound the bias and variance. Denote $e_i(A)$ to be the function that computes the $i^{th}$ eigenvalue. Then 
\begin{equation}
    e_{i}(A^{T}A+ \alpha I) = \tilde \sigma_{i}^2 + \alpha.
\end{equation}
This implies 
\begin{equation}
    g_{\alpha}(\tilde \sigma_i^2) = \frac{1}{\tilde \sigma_{i}^2+ \alpha}.
\end{equation}
Then rewriting to arrive at the filter function $\gamma_i^{(\tau_n)}$ as defined in \cite{BlanchardHoffmannReiss}
\begin{align}
\tilde \sigma_{i}^{2}g_{\tau_{n}^{-2}}(\tilde \sigma_{i}^2) &= \frac{\tilde \sigma^{2}_{i}}{\tilde \sigma_{i}^{2}+ \tau_{n}^{-2}} \\
&\implies \gamma_{i}^{(\tau_n)} = (1 + (\tau_{n}\tilde \sigma_{i})^{-2})^{-1}. 
\end{align}

\begin{remark}
\label{rem:qualification}
    We re-parameterize the filter function $g_\alpha$ as $\gamma_{i}^{(\tau_n)}$ in alignment of \cite{BlanchardHoffmannReiss}. The motivation for working with this formation is as follows. 
    \begin{enumerate}
        \item With $g(\tau_n , \widetilde \sigma) = \frac{\widetilde \sigma^2}{\widetilde \sigma^2 + \tau_n^{-2}}$, we have that $\theta^\dagger - \widetilde{\theta}_{\tau_n^{-2}} = 1 -g(\tau_n , \widetilde \sigma)$. (Compare definition of $r_\alpha$ in equation 4.8 in \cite{EngHanNeu96}.)
        \item  We thus have a way to write out the bias (variance) then in terms of the effect of the filter function for each of the $i^{th}$ components, which is denoted by $\gamma_{i}^{(\tau_n)}$. 
        \item The qualification index is the largest $q$ such that 
        \begin{equation}
        \label{def:qualification_index}
            \widetilde{\sigma}_i^q|1 -g(\tau_n , \widetilde \sigma) | = O((\tau_n^{-2})^q)
        \end{equation}
        holds. (See equations 4.24, 4.27 and 4.49 in \cite{EngHanNeu96}). And as $
         \widetilde{\sigma}_i \asymp i^{-p}$ for $p >0$ we have that $ \widetilde{\sigma}_i \leq 1$ for all $i$, and so \cref{def:qualification_index} can we written equivalently as the largest  $q$ such that 
        \begin{equation}
        \label{def:eqv_qualification_index}
            1 -g(\tau_n , \widetilde \sigma)  = O((\tau_n^{-2})^q)
        \end{equation}
        holds.
    \end{enumerate}
\end{remark}
We now repeat the following assumptions on the filter function found in \cite{BlanchardHoffmannReiss} (see Assumption R). 
\begin{assump}
\label{ass:1}
    Denote the regularisation function by $g(t, \sigma_i)$ where $g(t, \sigma_i) : \mathbb{R}_{+} \times \mathbb{R}_{+}  \rightarrow [0,1]$. Then the following must hold for $g(t,\sigma_i)$ to be in the class of regularising functions considered. 
        \begin{enumerate}
            \item The function $g(t, \sigma)$ is non-decreasing in $t$ and $\sigma$, continuous in $t$ with $g(0,\sigma)=0$ and $lim_{t \rightarrow \infty} g(t,\sigma) = 1$ for any fixed $\sigma >0$. 
            \item For all $t \geq t^\prime \geq t_0$. the function $\lambda \mapsto \frac{1-g(t^\prime, \sigma)}{1-g(t, \sigma)}$ is non-decreasing. 
            \item There exist positive constants $\rho, \beta_{-}, \beta_{+}$ such that for all $t \geq t_0$, and $\sigma \in [0,1]$ we have that 
            \begin{equation*}
                \beta_{-} \min ((t\sigma)^\rho,1) \leq g(t,\sigma) \leq \min (\beta_{+} (t\sigma)^\rho, 1).
            \end{equation*}
        \end{enumerate}
    \end{assump}
    \begin{remark}
           \begin{equation}
               g(\tau_{n}, \tilde \sigma) = (1 + (\tau_{n}\tilde \sigma)^{-2})^{-1}. 
            \end{equation}
    satisfies \cref{ass:1} with $\rho=2$ and $\beta_{-},\beta_{+} =1$. We check that indeed condition $3$ is satisfied. We see that 
        \begin{equation}
             g(\tau_{n}, \tilde \sigma) = (1 + (\tau_{n}\tilde \sigma)^{-2})^{-1} \leq \min \left(\tau_{n}\tilde \sigma)^2, 1 \right)
        \end{equation}
        so then the condition 3 is satisfied with $\rho =2$ and $\beta_{-},\beta_{+}=1$. Conditions $1,2$ were shown in \cite{BlanchardHoffmannReiss}.
    \end{remark}
We now use the proof techniques found in \cite{EngHanNeu96} and \cite{BlanchardHoffmannReiss}. As we have shifted our problem back to the simple Tikhonov, we can construct a result like those found in \cite{EngHanNeu96} Chapter 4 and Corollary 3.7 in \cite{BlanchardHoffmannReiss}. To do this, we use the bias variance decomposition of the integrated mean squared error and bound the bias and variance in terms of $\tau_n$, see \cref{lem:bias_bound_shifted}, and \cref{lem:var_bound_shifted}. From the structural assumption on $L$, we can transfer these results back to the original problem, see \cref{lem:bias_bound_org}, and \cref{lem:var_bound_org}. Finally, with careful choice of the prior covariance, we see that the error bound for estimating $\tilde \theta^\dagger$ is actually optimal.

Using the bias and variance decomposition, we can write the integrated mean square error of the estimator $\tilde \theta$ for $\tilde \theta^{\dagger}$ as
\begin{equation}
\label{eq:integrated_mse}
   \epsilon_n^2 := \mathbb{E} \left[\|\tilde \theta_{\tau_{n}} - \tilde \theta^{\dagger} \|^{2}\right] = \tilde B^{2}_{\tau_{n}} + \tilde V_{\tau_{n}}. 
\end{equation}
We see that they are functions depending on $\tau_n$. 
Using the filter function, $\gamma_{i}^{(\tau_{n})}$ we can write the bias as
\begin{equation}
\label{eq:bias}
    \tilde B_{\tau_{n}}^{2}(\tilde\theta^{\dagger}) = \sum_{i=1}^{D(n)} \left(1 - \gamma_{i}^{(\tau_{n})}\right)^2\tilde \theta_{i}^{2}
\end{equation}
where $\tilde \theta_i$ is the coefficients of $\tilde \theta$ wrt to the basis from $G^TG$. The variance can be written as
\begin{equation}
    \label{eq:var}
    \tilde V_{\tau_{n}} = \delta^{2}\sum_{i=1}^{D(n)} \left(\gamma_{i}^{(\tau_{n})}\right)^{2} \tilde \sigma_{i}^{-2}.
\end{equation}
The bias and variance in the observation space can be written as
\begin{equation}
    \label{eq:weak_bias}
    \tilde B_{\tau_n , \tilde \sigma}^2 =  \sum_{i=1}^{D(n)} \left(1 - \gamma_{i}^{(\tau_{n})}\right)^2 \tilde \sigma^2 \tilde \theta_{i}^{2}
\end{equation}

\begin{equation}
    \label{eq:weak_var}
    \tilde V_{\tau_n , \tilde \sigma} = \delta^2 \sum_{i=1}^{D(n)}\left(\gamma_i^{(\tau_n)}\right)^{2}.
\end{equation}
In \cite{BlanchardHoffmannReiss}, this is called the weak bias and variance. We now define the following stopping times:
\begin{align}
    \tau_{n,\mathfrak{w}} &= \tau_{n,\mathfrak{w}(\theta^\dagger} = \underset{}{\text{inf}} \left\{\tau_n \geq \tau_{0} : \tilde B_{\tau_n,\lambda}^{2}(\theta^\dagger) \leq \tilde V_{\tau_n,\lambda} \right\} \\
    \tau_{n, \mathfrak{s}} &= t_{n,\mathfrak{s}(\theta^\dagger)} = \underset{}{\text{inf}} \left\{\tau_n \geq \tau_{0} : \tilde B_{\tau_n}^{2}(\theta^\dagger) \leq \tilde V_{\tau_n} \right\}
\end{align}
where $\tau_0 \geq 0$ is an initial starting point. We also write $\tau^*_n$ to be the $\tau_n$ that balances the oracle, i.e the optimal $\tau_n$, when knowing the noise level. 
\begin{lemma} \cite{BlanchardHoffmannReiss}
\label{lem:upper_bound_obs_err}
For $\hat{\theta}^{\tau_{n,\mathfrak{w}}}$ an estimator for $\theta^\dagger$ whose bias and variance can be written as above in terms of a filter function $\gamma_i^{(\tau_n)}$, we have the following. 
\begin{equation}
       \mathbb{E}\left[\|\hat{\theta}^{(\tau_{n,\mathfrak{w}})} - \theta^\dagger\|_{A}^{2}\right] \leq 2 \underset{\tau_n \geq \tau_0}{\rm{inf}}\mathbb{E}\left[\|\hat{\theta}^{(\tau_n)} - \theta^\dagger\|_{A}^{2}\right].
\end{equation}
\end{lemma}
\begin{remark}
In \cite{BlanchardHoffmannReiss}, they show that the estimator $\hat{\theta}^{(\tau_{\rm dp})}$, which minimises \eqref{eq:loss_functional} with no regularisation operator and regularisation parameter $\tau_{\rm dp}$ achieves the minimax rate for signals in Sobolev balls. A consequence of Theorem 3.5 in \cite{BlanchardHoffmannReiss} see \cref{sec:appendix}  is that 
    \begin{equation}
    \label{rem:on_3.5}
        \mathbb{E}\left[\|\hat{\theta}^{(\tau_{\rm dp})} - \theta^\dagger\|^2\right] \lesssim \max\left(\mathbb{E}\left[\|\hat{\theta}^{(\tau_{n,\mathfrak{s}})} - \theta^\dagger\|^2\right], t_{n, \mathfrak{w}}^2 \mathbb{E}\left[\|\hat{\theta}^{(t_{n, \mathfrak{w}})} - \theta^\dagger\|^2_A\right]\right).
    \end{equation}  
\end{remark}
\begin{remark}
\label{rem:2_on_3.5}
    Using \cref{lem:upper_bound_obs_err} and \cref{rem:on_3.5} we have that 
    \begin{equation}
        \mathbb{E}\left[\|\hat{\theta}^{(\tau_{\rm dp})} - \theta^\dagger\|^2\right] \lesssim \underset{\tau_n \geq \tau_0}{\rm{inf}}\max\left(\mathbb{E}\left[\|\hat{\theta}^{(\tau_{n, \mathfrak{s}})} - \theta^\dagger\|^2\right], \tau_n^2 \mathbb{E}\left[\|\hat{\theta}^{(\tau_n)} - \theta^\dagger\|^2_A\right]\right).
    \end{equation}
\end{remark}
We will now bound the bias, the variance and the weak variance in terms of $\tau_n$. 
\begin{lemma}
\label{lem:bias_bound_shifted}
    Given the \cref{ass:1}, we can bound 
    \begin{equation}
    \label{eq:qual}
        1 - g(\tau_{n}, \tilde \sigma) \leq C_{q}(\tau_{n} \tilde \sigma)^{-2q}
    \end{equation}
where $q$ is the qualification index (see \cref{rem:qualification}). If  $q > \beta /(1 +  2\alpha + 2p)$, we can bound the bias by 
\begin{equation}
\label{eq:shifted_bias_bound}
    \tilde B_{\tau_{n}}^{2}(\tilde \theta^{\dagger}) \leq C_{q}c_{A}^{-2 \beta / (1/2 + \alpha +p)} \|\tilde \theta \|_{\ell^2(\mathbb{N})}^{2} \tau_{n}^{-2 \beta / (1/2 +  \alpha + p )}
\end{equation}
\end{lemma}
\begin{proof}
 For $q > \beta /(1 +  2\alpha + 2p)$ we have that 
    \begin{align}
    \tilde B_{\tau_{n}}^{2} (\theta) &= \sum_{i=1}^{D}(1-\gamma_{i}^{(\tau_{n})})^{2}\tilde \theta_{i}^{2} \nonumber \\
    &\leq \sum_{i}^D \left(C_{q}(\tau_{n} \tilde \sigma_{i})^{-2q}\right)^{2} \tilde \theta{i}^{2} \nonumber \\
    &\leq C_{q}^{2}\sum_{i}^{D} \left(\tau_n^{-2q} \tilde\sigma_{i}^{-2q}\right)^{2}\tilde \theta_{i}^{2} \nonumber \\
    & \leq  C_{q}^{2}\sum_{i}^{D} (\tau_{n}\tilde \sigma_{i})^{-4q}\tilde \theta_{i}^{2} \nonumber \\
     & \leq  C_{q}^{2}\sum_{i}^{D} (\tau_{n}\tilde \sigma_{i})^{-(2\beta / (p + 1/2 + \alpha)}\tilde \theta_{i}^{2} \nonumber \\
    &\leq C_q^2 c_{A}^{-2 \beta / (1/2 + \alpha + p)}  \tau_{n}^{-2 \beta / (p + 1/2 + \alpha)}  \|\tilde \theta\|^2_{\ell^2(\mathbb{N})}
    \end{align}
\end{proof}

\begin{lemma}

\label{lem:weak_bias_bound_shifted}
    Given the \cref{ass:1}, we can bound 
    \begin{equation}
        1 - g(\tau_{n}, \tilde \sigma) \leq C_{q}(\tau_{n} \tilde \sigma)^{-2q}
    \end{equation}
where $q$ is the qualification index (see \cref{rem:qualification}). If $2q - 1 > \beta /(\frac{1}{2} +  \alpha + p)$ we can bound the weak bias by 
\begin{equation}
\label{eq:shifted_weak_bias_bound}
   \tau_n^2 \tilde B_{\tau_{n}}^{2}(\tilde \theta^{\dagger}) \leq C_{q}c_{A}^{-2 \beta / (1/2 + \alpha +p)} \|\tilde \theta \|_{\ell^2(\mathbb{N})}^{2} \tau_{n}^{-2 \beta / (1/2 +  \alpha + p )}
\end{equation}
\end{lemma}

\begin{proof}
As $2q - 1 > \beta /(\frac{1}{2} +  \alpha + p)$
    \begin{align}
        \tau_n^2 \tilde B_{\tau_n , \tilde \sigma}^2 &=  \tau_n^{2}\sum_{i=1}^{D}(1 -\gamma_{i}^{(\tau_n)})^{2}\tilde \sigma_{i}^{2}\sigma \theta_{i}^{2} \nonumber \\
        & \leq \sum_{i=1}^{D} C_{q}^2\text{min}((t \lambda_{i})^{-2q + 1}, t \lambda_{i})^{2} \nonumber \\
        & \leq C_q^2 c_{A}^{-2 \beta / (\frac{1}{2} + \alpha + p)} \tau_{n}^{-2\beta / (p + 1/2 + \alpha)} \|\tilde \theta\|^2_{\ell^2(\mathbb{N})}
    \end{align}
\end{proof}

\begin{lemma}
\label{lem:bias_bound_org}
    We can transfer the (weak) bias bound for estimating $\tilde \theta^\dagger$ to a bound for the (weak) bias of estimating $\theta^\dagger$. That is 
    \begin{equation}
        B_{\tau_n}^2 \leq C_{L,1} \tilde  B_{\tau_n}^2 \quad  \tau_n^2B_{\tau_n, \tilde \sigma}^2 \leq C_{L,2} \tilde  \tau_n^2 B_{\tau_n, \tilde \sigma}^2 
    \end{equation}
    where $C_L$ is some constant not depending on $\tau_n$. Moreover,
    \begin{equation}
    \label{eq:bias_bound}
        B_{\tau_{n}}^{2}(\theta^{\dagger}) \leq ||L^{-1}||_{\mathbb{T}}^{2}C_{q}c_{A}^{-2 \beta / (p + 1/2 + \alpha)} \|\tilde \theta \|_{\ell^2(\mathbb{N})}^{2} \tau_{n}^{-2 \beta / (1/2 +  \alpha + p )}
    \end{equation}
\end{lemma}
\begin{proof}
    Recall that 
    \begin{equation}
        \tilde \theta = L\theta.
    \end{equation}
    Denote the minimiser of the shifted regularisation problem \eqref{eq:shifted_minproblem} with regularisation parameter $\tau_{\rm dp}$ determined by stopping according to \eqref{eq:dp_stopping_time} by 
    \begin{equation}
    \label{eq:tilde_estimator}
        \tilde \theta_{\tau_{\rm dp}}. 
    \end{equation}
We have the following: 
\begin{align}
     B_{\tau_{n}}^{2} &= \| \theta^{\dagger} - \hat \theta_{\tau_{\rm dp}} \|_{\ell^2(\mathbb{N})} \nonumber \\
     &= \|L(L^{-1}\theta^{\dagger} - \hat \theta_{\tau_{\rm dp}})\|_{\ell^2(\mathbb{N})} \nonumber  \\
     &=\|L^{-1}(\tilde \theta^{\dagger} - \tilde \theta_{\tau_{\rm dp}})\|_{\ell^2(\mathbb{N})}\nonumber \\
     &= \|L^{-1}\|_{\mathbb{T}}\|\tilde \theta^{\dagger} - \tilde  \theta_{\tau_{\rm dp}}\|_{\ell^2(\mathbb{N})}\nonumber  \\
     & = C_L \tilde  B_{\tau_n}^2 
\end{align}
The proof for the weak bias follows the same. 
\end{proof}
We will now bound the variance of the estimator. 
\begin{lemma}
The variance \eqref{eq:var} of $\tilde \theta_{\tau_n}$ is upper bounded by
\label{lem:var_bound_shifted}
\begin{equation}
\label{eq:shifte_varbound}
    \tilde V_{\tau_{n}}  \lesssim \delta^2 \tau_{n}^{2+1 /(p + 1/2 + \alpha)}.
\end{equation}
\end{lemma}
\begin{proof}
    \begin{align}
        \tilde V_{\tau_{n}} &= \delta^{2}\sum_{i=1}^{D}  \left( \gamma_{i}^{(\tau_{n})} \right)^{2}\tilde \sigma_{i}^{-2} \nonumber \\
        &\leq \delta^{2}\sum_{i=1}^{D}  \text{min } ((\tau_{n} \tilde \sigma_{i})^{2\rho},1 )\tilde \sigma_{i}^{-2} \nonumber \\
        &\leq \delta^{2}\sum_{i=1}^{D} (\tau_{n} \tilde \sigma_{i})^{2\rho}\tilde \sigma_{i}^{-2} \nonumber \\
        &\leq \delta^{2}\sum_{i=1}^{D} \tau_{n}^{2\rho} \tilde \sigma_{i}^{2\rho}\tilde \sigma_{i}^{-2}\nonumber \\
        & \leq \delta^{2}\sum_{i=1}^{D} \tau_{n}^{2\rho} \tilde \sigma_{i}^{2\rho -2} \nonumber \\
        &\lesssim \delta^2 \tau_{n}^{2+1 /(p + 1/2 + \alpha)}
        \end{align}
    for $2\rho -2 > 1/(p + 1/2 + \alpha)$. The first inequality follows from \cref{ass:1}. 
\end{proof}
We now bound the weak variance of the estimator,
\begin{lemma}
The weak variance \eqref{eq:weak_var} of $\tilde \theta_{\tau_n}$ is upper bounded by 
 \label{lem:weak_varbound}
    \begin{equation}
    \tilde V_{\tau_{n}, \tilde \sigma}  \lesssim \delta^2 \tau_{n}^{1 /(p + 1/2 + \alpha)}.
\end{equation}
\end{lemma}
\begin{proof}
    \begin{align}
        \tilde V_{\tau_{n},\tilde \sigma} &= \delta^{2}\sum_{i=1}^{D}  \left( \gamma_{i}^{(\tau_{n})} \right)^{2} \nonumber \\
        &\leq \delta^{2}\sum_{i=1}^{D}  \text{min } ((\tau_{n} \tilde \sigma_{i})^{2\rho},1 )\nonumber  \\
        &\leq \delta^{2}\sum_{i=1}^{D} (\tau_{n} \tilde \sigma_{i})^{2\rho}\nonumber  \\
        & \leq \delta^{2}\sum_{i=1}^{D} t^{2\rho} \tilde \sigma_{i}^{2\rho} \nonumber \\
        &\lesssim \delta^2 \tau_{n}^{1 /(p + 1/2 + \alpha)}
        \end{align}
    for $2\rho > 1/(p + 1/2  \alpha + p)$.  The first inequality follows from \cref{ass:1}. 
\end{proof}
\begin{lemma}
\label{lem:var_bound_org}
    The (weak) variance of the original problem is of the same order. That is 
    \begin{equation}
                V_{t} \leq C_L^2 \tilde V_{t} \quad \quad V_{t,\sigma} \leq C_L^2 \tilde V_{t,\tilde \sigma}.
    \end{equation}
\end{lemma}
\begin{proof}
    By basic algebraic properties of the variance, if $x$ is a random vector with variance $\Sigma$ and $L$ a matrix, then $\text{Var}(LX)=L^T\Sigma L$. So then $\text{var}(Lx)=L\text{Var}(x)L^T$. As $L$ does not depend on $\tau_{n}$, then the $\text{Var}(\tilde \theta_{\tau_{\rm dp}})=L\text{var}(\hat \theta_{\tau_{\rm dp}})L^{T}  \lesssim \delta^2 \tau_{n}^{2+1 /(1 + 2\alpha + 2p)}$. The same holds in the observation space. 
\end{proof}

\begin{lemma}
\label{lem:optimal_t}
The optimal scaling of $\tau_n$ is of order $n^{(2 p + 1 + 2 \alpha ) / (4\beta + 4p + 4 + 4 \alpha) }$ if $\beta < 1  +2 \alpha + 2p$. 
\end{lemma}
\begin{proof}
    Setting, 
    \begin{equation}
        \epsilon_{n}^{2} = \tau_{n}^{2 \beta / (p + 1/2 + \alpha)} + \delta^{2}\tau_{n}^{2+1/(1/2 +  \alpha + p)} = 0
    \end{equation}
    We find that 
   \begin{equation}
        \tau_{n} = \delta^{-(2p + 1 + 2 \alpha) / 2 \beta  + 2p + 2 \alpha + 2} = n^{(2 p + 1 + 2 \alpha ) / (4\beta + 4p + 4 + 4 \alpha) }.
   \end{equation}
\end{proof}

\begin{lemma}
\label{lem:rate_shifted_problem}
Recall that 
\begin{equation*}
   \epsilon_n^2 := \mathbb{E} \left[\|\tilde \theta_{\tau_{n}} - \tilde \theta^{\dagger} \|^{2}\right] = \tilde B^{2}_{\tau_{n}} + \tilde V_{\tau_{n}}. 
\end{equation*} Using the bias and variance bounds from above, we have that:

    $$
    \epsilon_n^2 \lesssim {n}^{-4\beta / (2\beta  + 2p +  2\alpha + 2)}.
    $$
\end{lemma} 

\begin{proof}
    Using the optimal $\tau_n$ found in \cref{lem:optimal_t}, have that 
    $$
    \epsilon_{n}^2= \left(n^{(2 p + 1 + 2 \alpha ) / (4\beta + 4p + 4 + 4 \alpha) }\right)^{-2 \beta / (p + 1/2 + \alpha)} + \frac{1}{n}\left(n^{(2 p + 1 + 2 \alpha ) / (4\beta + 4p + 4+ 4 \alpha) }\right)^{2 +1/(p + 1/2 + \alpha)}
    $$
    $$
    \epsilon_n^2 \lesssim {n}^{-4\beta / (2\beta  + 2p +  2\alpha + 2)} 
    $$
\end{proof} 
\begin{theorem} 
\label{thm:minimax}
      Provided that $\kappa_n  \asymp {D(n)}/n $ in (\ref{eq:dp_stopping_time}),  the truncation dimension $D(n)$ in (\ref{eq:discrepany_proj}) satisfies (\ref{eq:trunctation_dimension}),   
      \begin{equation} \label{eq:oversmoothing}
      \beta < 2\alpha + 2p + 1,
      \end{equation}
      and 

      \begin{equation}
          \label{eq:critical_value}
          (\delta^{-2/(2 \beta + 2p + 2\alpha + 2)})\geq \sqrt{D(n)},
      \end{equation}
      It follows that
    \begin{equation} \label{eq:rate}
        \sup_{\theta^\dagger \in S^\beta(1)} \mathbb{E}||\widehat{\theta}_{\tau_{\rm dp}} - \theta^\dagger ||_{}^2  \lesssim {n}^{-4\beta / (2\beta  + 2p +  2\alpha + 2)} 
    \end{equation}
    in the limit $D(n) \rightarrow \infty$, where $\widehat{\theta}_{\tau_{\rm dp}}$ is the estimator (\ref{eq:Bayes_mean}).
\end{theorem}
\begin{proof}
    In the above lemmas, we bound the bias and variance in terms of $\tau_n$ for the reparameterized problem \cref{eq:approximation_error1}, and compute the optimal $\tau_n$.  Thus, we have that 
    \begin{equation}
        \underset{\tau_n \geq \tau_0}{\rm{inf}}\max\left(\mathbb{E}\left[\|\hat{\tilde\theta}^{(\tau_{n})} - \tilde\theta^\dagger\|^2\right], \tau_n^2 \mathbb{E}\left[\|\hat{\theta}^{(\tau_n)} - \theta^\dagger\|^2_A\right]\right) \leq \epsilon_n^2
    \end{equation}
    where  $\epsilon_n \asymp n{^{-\beta /(\beta + \alpha + p + 1)}}$ in the case when $ \beta < 2\alpha + 2p + 1$ . The forward operator in  \eqref{eq:shifted_minproblem} satisfies \cref{ass:1}, and the eigenvalues of $A$ decay polynomially. Using that the optimal $\tau_n  \gtrsim n^{(2 p + 1 + 2 \alpha ) / (4\beta + 4p + 4 + 4 \alpha) }$ we have that 
    \begin{align}
        V_{\tau_n ,\tilde \sigma} & \sim \delta^2 \left(n^{( 2p + 1 +  2\alpha ) / (4\beta + 4p + 4 + 4 \alpha) }\right)^{1/(p + 1/2 + \alpha)}\nonumber  \\
        &\sim \delta^2 \left(\delta^{-2 / (2\beta + 2p + 2\alpha + 2)}\right) \gtrsim \delta^2 \sqrt{D}
    \end{align}
    By \cref{rem:2_on_3.5} we have then that  

    \begin{align}
    \label{eq:shifted_err_bound}
       \mathbb{E}||\widehat{\tilde \theta}_{\tau_{\rm dp}} - \tilde \theta^\dagger ||^2 &\lesssim  \underset{\tau_n \geq \tau_0}{\rm{inf}}\max\left(\mathbb{E}\left[\|\hat{\tilde\theta}^{(\tau_{n})} - \tilde\theta^\dagger\|^2\right], \tau_n^2 \mathbb{E}\left[\|\hat{\tilde \theta}^{(\tau_n)} - \tilde \theta^\dagger\|^2_A\right]\right) \nonumber  \\
       & \lesssim \epsilon_n^2\nonumber  \\
       &={n}^{-4\beta / (2\beta  + 2p +  2\alpha + 2)} .
    \end{align}
     Then by \cref{lem:bias_bound_org} and \cref{lem:var_bound_org}, we saw that the bounds in terms of $\tau_n$ were of the same order and were only multiplied by constants. That is, 
     \begin{align}
         \mathbb{E}||\widehat{\theta}_{\tau_{\rm dp}} - \theta^\dagger ||_{}^2 &= B_t^2 + V_t \nonumber \\
         &\lesssim \tilde B_t^2 + \tilde V_t \nonumber \\
         &=\mathbb{E}||\widehat{\tilde \theta}_{\tau_{\rm dp}} - \tilde \theta^\dagger ||^2. 
     \end{align}
     Thus, we have that 
        \begin{align}
    \label{eq:org_err_bound}
        \mathbb{E}||\widehat{\theta}_{\tau_{\rm dp}} - \theta^\dagger ||_{}^2 &\lesssim \mathbb{E}||\widehat{\tilde \theta}_{\tau_{\rm dp}} - \tilde \theta^\dagger ||^2\nonumber  \\
        &\lesssim  \underset{\tau_n \geq \tau_0}{\rm{inf}}\max\left(\mathbb{E}\left[\|\hat{\tilde\theta}^{(\tau_{n})} - \tilde\theta^\dagger\|^2\right], \tau_n^2 \mathbb{E}\left[\|\hat{\tilde \theta}^{(\tau_n)} - \tilde \theta^\dagger\|^2_A\right]\right)\nonumber   \\
       & \lesssim \epsilon_n^2 \nonumber \\
      & \lesssim  {n}^{-4\beta / (2\beta  + 2p +  2\alpha + 2)}  .
    \end{align}
 \end{proof}
\noindent
 We then have that the Bayesian estimator, as given by the mean of the stopped posterior, achieves the frequentist minimax rate for $\widetilde{\theta}$ \cref{def:tilde_theta}, which is suboptimal for $\theta$.

 \begin{remark}
     We remark that the rate of contraction is $\epsilon_n^2 \asymp{n}^{-4\beta / (2\beta  + 2p +  2\alpha + 2)} $. This holds only when $\beta \leq 2\alpha + 2p +1$. We can write the rate as  
     \begin{equation}
         \epsilon_n^2 \asymp {n}^{-4\beta / (2\beta  + (2p +  2\alpha + 1) + 1)} 
     \end{equation}
    and recalling that the rate of decay of the singular values of $(A^TA)$ was $(2p+1/2 + \alpha)$, we see that this is the optimal rate for $\tilde{\theta}$.
 \end{remark}

 We want to now extend \cref{thm:minimax}, which is a statement only for the MAP estimator.  We will show further that the posterior contracts at the same rate. We first state the following definition: 
 \begin{definition}
    A sequence $(\epsilon_n)_n$ of positive numbers is a posterior contraction rate at the  parameter $\theta^\dagger$ wrt to the semi-metric $d$ if for every sequence $(M_n)_n \rightarrow \infty$, it holds that,
    \begin{align}
        \label{eq:contraction_rate}
        \Pi_n (\theta : d(\theta, \theta^\dagger) \geq M_n \epsilon_n \mid Y) \overset{P_{\theta^\dagger}}{\longrightarrow} 0 
    \end{align}
    as $n \rightarrow \infty$. Where $\Pi_n(\cdot \mid Y)$ is the posterior given observations $Y$ and given prior $\Pi_n$.
\end{definition}

\noindent
Intuitively, $(\epsilon_n)_n$ is the rate at which the $D(n)$ ball of radius $M_n \epsilon_n$ decreases such that ``most'' of the posterior mass is inside this ball.

\begin{cor}
 \label{cor:1}
     If $\kappa \asymp D(n) / n $, then $\mathbb{P}_Y( \tau_{\rm lo} \leq \tau_{\rm dp}) \rightarrow 1$, where $\tau_{\rm lo}, \asymp   n^{(2 p + 1 + 2 \alpha ) / (2\beta + 2p + 2 + 2\alpha) }$, and  the adaptation interval is $[0, \beta_+]$,  such that $\beta < \beta_+ \leq 1 + 2p + 2 \alpha.$ 
 \end{cor}
 
 \begin{proof}     
     If $\kappa \gtrsim D(n) /n$, then $\mathbb{P}_Y(\tau_{\rm dp} \geq \tau_{\rm lo})\rightarrow 1$ as asymptotically $R_{\tau_{\rm dp}} \gtrsim \kappa$ and by the monotonicity of the bias. We further conclude that the adaptation interval, is $[\beta_{-}, \beta_{+}]$ where $\beta_{-} =0$, as the estimator for the effective dimension was chosen as $\rho^{-1/(2p+1)}$ implying $\beta=0$. We then have, that $\beta_{+} = 1 + 2p + 2\alpha$, as only in this case can the prior for fixed $\alpha$ be rescaled and achieve optimal rates, see proof of Theorem 4.1 in \cite{knapik2011}. 
 \end{proof}

 \noindent
 We will further show that as the truncation dimension $D(n) \rightarrow \infty$, the posterior dependent on $\tau_{\rm dp}$ also contracts at the rate $\epsilon_n^2 \asymp {n}^{-4\beta / (2\beta  + 2p +  2\alpha + 2)}$. We use the results from Theorem \ref{thm:minimax}, and only need to consider the behaviour of the tails, already knowing the rate at which the mean of the posterior is moving towards the truth.

\begin{theorem}
\label{thrm:contraction_rate}
    Let $G^\ast G$ and $C_0$ have the same eigenfunctions. Denote the eigenvalues of $G^\ast G$ and $C_0$ by $\sigma_i^2$ and, $\lambda_i$ respectively. Recall the structural assumptions, 
    \begin{equation}
        \sigma_i \asymp i^{-p} 
    \end{equation}
    and entry-wise prior
    \begin{equation}
        \theta_i \sim \mathcal{N}(0,\tau_n^2 \lambda_i). \quad \lambda_i \asymp i^{-1-2\alpha}.
    \end{equation}
    Denote the posterior associated to the estimated stopping time  $\tau_{\rm dp }$ by $\Pi_{n,\tau_{\rm dp }}(\cdot \mid Y)$. Then 
    \begin{equation}
       \Pi_{n,\tau_{\rm dp }} \left( \widehat{\theta}_{\tau_{\rm dp }} : || \widehat{\theta}_{\tau_{\rm dp }} - \theta^\dagger||_{\ell^2(\mathbb{N})} \geq M_n \epsilon_n\right) \overset{\mathbb{P_{\theta^\dagger}^{(n)}}}{\rightarrow} 0
    \end{equation}
    for every $M_n \rightarrow \infty$, and with $\epsilon_n = {n}^{-\beta / (\beta  + p +  \alpha + 1)} $.
\end{theorem}
\begin{proof}

\begin{equation}
\label{eq:bayes_risk}
    \int || \theta - \theta^\dagger ||^2_{1} d \Pi_n (\theta \mid Y) = ||K_{\tau_n}Y- \theta^\dagger||_{1}^2 + \text{tr}(C_{\tau_n})
\end{equation}
where we recall that $K_{\tau_n}$ is defined in \cref{prop:bayesian_setup}

Then by Markov's inequality 

\begin{align*}
    \Pi_n(||\theta-\theta^\dagger)||_{1} \geq M_n\epsilon_n \mid Y) \leq \frac{\int || \theta - \theta^{\dagger} ||^2_{1} d \Pi_n (\theta \mid Y)}{M_n^2\epsilon_n^2}
\end{align*}
\begin{enumerate}

    \item $\int || \theta - \theta^\dagger ||^2_{1} d \Pi_n (\theta \mid Y) = \sum_i^{D(n)} \left(\left(\mathbb{E}[\theta_i \mid Y]\right)\right)^2 + \rm{Var}(\theta_i \mid Y)$
    \item $\mathbb{E} ||K_{\tau_n}Y - \theta^\dagger||^2_1  = ||K_{\tau_n}G\theta^\dagger - \theta^\dagger||^2_1 + \frac{1}{n} \rm{tr}(K_{\tau_n} K_{\tau_n}^T)$ 
    \item $\rm{Var}(\theta \mid Y) = \rm{Var}(K_{\tau_n}Y) = \frac{1}{n} tr(K_{\tau_n} K_{\tau_n}^T)$
\end{enumerate}

 We have that from \cref{thm:minimax} we have that $(2)$ is bounded by $ \epsilon_n = n^{-\beta / (p +1/2 + \alpha)}$ when $\tau_{\rm dp}$ is at least of order $n^{2 \alpha + 2p +1 / (4 \beta + 4 \alpha + 4p + 4)}$. To bound the Bayes risk \eqref{eq:bayes_risk}, we need to upper bound the posterior spread by $\epsilon_n$. The posterior spread for each time $\tau_n$ is given by
\begin{equation}
    C_{\tau_{n}} := \tau_{n}^2 C_0 - \tau_{n}^2 K_{\tau_{n}} \left( G C_0 G^T + \frac{1}{ \tau_{n}^2}I \right)K_{\tau_{n}}^T. 
\end{equation}
We will bound the trace of the posterior covariance as done in \cite{knapik2011}. 

\begin{align}
     \label{eq:seqspace_posterior_variance}
      \rm{tr}(C_{\tau_n})
            &= \sum_i^{D(n)} \frac{ \lambda^{\tau_n}_i  }{1 + \lambda^{\tau_n}_i\sigma_{i}^2} \nonumber \\
			&\asymp \sum_{i}^{D(n)} \frac{\tau_{n}^{2}i^{-1-2 \alpha}}{1+  \tau_{n}^{2} i^{-1-2 \alpha-2p}}
 \end{align}
 where $\lambda^{\tau_n}$ are the eigenvalue of $C_{\tau_n}$. By applying \cref{lem:8.2} with $S(i)=1, q=-1/2, t=1+ 2 \alpha, u=1 + 2 \alpha + 2p, v=1, N= \tau_n^2$ that as $n \rightarrow \infty$ we have that 
\begin{align}
    \tau_n^2 \sum_i^{D(n)} \frac{i^{-1-2 \alpha}}{1 +  \tau_n^2 i^{-1 -2 \alpha - 2p}} &\leq  \tau_n^2 N^{-(t + 2q)/u} \left(N^{1/u}\right) \nonumber \\
    \tau_n^2 \sum_i^{D(n)} \frac{i^{-1-2 \alpha}}{1 +  \tau_n^2 i^{-1 -2 \alpha - 2p}} &\leq \tau_n^2 N^{-(1 + 2 \alpha -1)/1 + 2 \alpha + 2p }  N^{1/1 + 2 \alpha + 2p} \nonumber \\
    &\leq \tau_n^2 (\tau_n^2)^{- 2 \alpha /1 + 2 \alpha + 2p }\nonumber \\
    &\leq  (\tau_n^2)^{1 - 2 \alpha/1 + 2 \alpha + 2p }
\end{align}

 so then 
   \begin{equation}
     ||C_{\tau_n}||^{1/2}_{\mathbb{T}} \lesssim (\tau_n^2)^{1-2\alpha / 1 + 2 \alpha + 2p}.
 \end{equation}
 Then we have that 

\begin{align}
    {n}^{-\beta / (\beta  + p +  \alpha + 1)}  ||C_{\tau_{\rm dp}}||^{-1/2}_{\mathbb{T}}  \geq {n}^{-\beta / (\beta  + p +  \alpha + 1)}  \left(n^{(2 p + 1 + 2 \alpha ) / (2\beta + 2p + 2 + 2\alpha) }\right)^{1 - 2\alpha / 1 + 2 \alpha + 2p} \rightarrow \infty
\end{align}
as $n \rightarrow \infty$. So then $||C_{\tau_{\rm dp}}||^{1/2}_{\mathbb{T}} \rightarrow o(\epsilon_n)$, and so 
    \begin{align}
        \underset{n \rightarrow \infty}{\rm lim}  \Pi_{n,\tau_{\rm dp}(n)}(||\widehat \theta_{\tau_n} -\theta^\dagger||^2_{\ell^2(\mathbb{N})} \geq \epsilon_n) = 0
    \end{align}
\end{proof}

\begin{remark}
   The posterior for $\widetilde{\theta}$ also contracts a rate $\epsilon_n \asymp{n}^{-\beta / (\beta  + p +  \alpha + 1)}$ . This can be seen from using \cref{lem:rate_shifted_problem}, and following the steps of the proof above. In this case, then the posterior contracts optimally for $\widetilde{\theta}$ \cref{def:tilde_theta}. 
\end{remark}
The next question we would like to address is whether the spread of the stopped posterior $\Pi_{n, \tau_{\rm dp}}$ has good frequentist coverage. First, we define a credible ball. 

\begin{definition}
    Denote the mean of the posterior by $KY$. Then the credible ball centred at $KY$ is defined as 
    \begin{equation}
    \label{eq:credible_set}
        KY + B(r_{n,c}) := \{\theta \in H_1 :||\theta - KY||_{H_1} < r_{n,c}  \}
    \end{equation}
    where $B(r_{n,c})$ is the ball centred at $0$ with radius $r_{n,c}, c \in (0,1)$ denotes the desired credible level of $1-c$. The $r_{n,c}$ is chosen such that 
    \begin{equation}
        \Pi_{n, \tau_n} (KY + B(r_{n,c}) \mid Y) = 1-c.
    \end{equation}
\end{definition}

\noindent
We can define the frequentist coverage of the set $KY + B(r_{n,c})$ as 
\begin{equation}
\label{eq:coverage}
    \Pi_{\theta^\dagger}(\theta^\dagger \in KY + B(r_{n,c})).
\end{equation}
\begin{cor}
    For fixed $\alpha > 0$, if $\theta^\dagger_i = C_i i^{-1-2\beta'}$ for all $i=1,...,D(n)$, then as $n \rightarrow \infty$, $ \Pi_{n, \tau_{\rm dp}}$  has frequentist coverage 1.
\end{cor}
\begin{proof}
    Recall that $\widetilde{\theta} \in H^\beta$ and that $\widetilde{\theta}^\dagger := C_0^{-1/2} \theta_0$. Then we have that $\beta' > \beta$. It was assumed that $\beta < 1 + 2p + 2\alpha$, so then $\tilde \beta = \beta \wedge 1 + 2p + 2 \alpha \implies \tilde  \beta = \beta$. As $\kappa \asymp D(n) /n$, then $\tau_{\rm dp} \gtrsim  n^{(2 p + 1 + 2 \alpha ) / (2\beta + 2p + 2 + 2\alpha) }$ by \cref{cor:1}. So then by \cref{thrm:coverage} part (1) for $\widetilde{\theta}^\dagger$ in the unit ball, we have that asymptotic coverage of the posterior for $\widetilde{\theta}^\dagger$ is 1. Furthermore then let $r^*_{n,c, \widetilde{\theta}^\dagger}$ be the correct radius such that 
    \begin{equation}
    \Pi_{\widetilde{\theta}^\dagger}(\widetilde{\theta}^\dagger \in KY + B(r^*_{n,c, \widetilde{\theta}^\dagger})) = 1-c.
    \end{equation}
    We have also by \cref{thrm:coverage} part (1) that 
    \begin{equation}
        r_{n,c} \asymp r^*_{n,c, \widetilde{\theta}^\dagger}
    \end{equation}
    As $\beta' > \beta$ we have that posterior rate $\epsilon_n \asymp n^{-\beta/ \beta + p + \alpha +1}$ is slower than the minimax rate for $\theta^\dagger$, and thus 
    \begin{equation}
            r^*_{n,c, \widetilde{\theta}^\dagger }>     r^*_{n,c, \theta^\dagger}
    \end{equation}
    and so then 
    \begin{equation}
    \Pi_{\theta^\dagger}(\theta^\dagger \in KY + B(r_{n,c})) \rightarrow 1
    \end{equation}
    as $r_{n,c} > r^*_{n,c, \theta^\dagger}$.
\end{proof}

\noindent
Moreover, it can be seen from conditions $(1,2)$ of Theorem 4.2, where we view $\tau_n$ as the time parameter, that it is important not to stop prematurely to have the desired frequentist coverage.

%
\section{Ensemble Kalman--Bucy inversion}
\label{sec:EKI}
%
In this section, we reformulate the Bayesian inverse 
problem defined by (\ref{eq:model}) and the Gaussian process prior
$\mathcal{N}(0,\tau_n^2 C_0)$ in terms of the time-continuous ensemble Kalman--Bucy filter \cite{reich10,CotterReich2013,CRS22}. To do so, we introduce a new time-like variable $t\ge 0$ and an $H_1$-valued and time-dependent Gaussian process denoted by $\Theta_t$. This process satisfies the McKean--Vlasov evolution equation
\begin{equation} \label{eq:DEnKBF}
\frac{{\rm d}\Theta_t}{{\rm d}t} = n \Sigma_t G^\ast  \left\{ Y - \frac{1}{2} G \left( \Theta_t + \widetilde{\theta}_t\right) \right\} , \qquad \Theta_0 \sim \mathcal{N}(0,C_0),
\end{equation}
where the mean and covariance operator of $\Theta_t$ are denoted by $\widetilde{\theta}_t$ and $\Sigma_t$, respectively. Even when $G$ is linear, the evolution equation is still nonlinear as it is dependent on the evolution of the mean and covariance of $\Theta_t$, which in turn depend on the law of $\Theta_t$. In the case when $G$ is a non-linear operator, one needs to check case specifically if an adjoint exists. In the situation where $G$ is a solution operator to a linear PDE, then such an adjoint can be defined depending on working in a Hilbert space and the boundary conditions. 

Although the evolution equations
(\ref{eq:DEnKBF}) Nonlinear, closed-form solutions in terms of the mean and the covariance operator are available and given by
\begin{subequations} \label{eq:moments1} 
\begin{align}
\widetilde{\theta}_t &= C_0 G^\ast \left(G  C_0 G^\ast + \frac{1}{t}I\right)^{-1} Y,\\
\Sigma_t &= C_0 - C_0 G^\ast \left(G C_0 
G^\ast + \frac{1}{t} I\right)^{-1} G C_0.
\end{align}
\end{subequations}
Comparison to the estimator (\ref{eq:Bayes_mean}) reveals that
\begin{equation}
\widetilde{\theta}_t = \widehat{\theta}_{\tau_{n}}
\end{equation}
for $t = \tau_{n}^2$. Furthermore, under the same relation between $t$ and $\tau_{n}$, it also holds that
\begin{equation}
C_{\tau_{n}} = t \Sigma_t .
\end{equation}
Hence, upon defining $D(n)$, $P_n$, and $\kappa$ as before, the evolution equations (\ref{eq:DEnKBF}) are integrated in time until
\begin{equation} \label{eq:stopping_EnKBF1}
    t_{\rm dp} := \inf\,\left\{
    t > 0: \|P_n(Y-G\widetilde{\theta}_t)\|^2_2 \le \kappa \right\}.
\end{equation}

We next propose an adapted ensemble implementation of the mean-field EnKBF formulation (\ref{eq:DEnKBF}). Such an idea has been proposed in \cite{Parzer}, where the ensemble size is grown with each iteration. We will keep our ensemble size fixed throughout the iterations. First note that by the Schmidt–Eckhardt–Young–Mirsky theorem for deterministic low rank approximation of self-adjoint trace class operators, truncating the SVD of $C_0$ at $J$ gives us the approximation error of  
\begin{equation} \label{eq:determin_approx}
\| V_{J-1} V^{T}_{J-1}- C_0\|_{\mathcal{L}(H_1,H_1)}  =
J^{-2 \alpha -1}.
\end{equation}
where $V_{J-1} V^{T}_{J-1} $ is the truncated SVD of $C_0$.
Let us generate $J$ independent $H_1$-valued random ensemble members $\Theta_0^{(i)}$, $i=1,\ldots,J$ 
such that their mean is zero and the covariance of each member is $C_0$. Denote their empirical mean by
\begin{equation}
    \widetilde{\theta}_0^J = \frac{1}{J}
    \sum_{i=1}^J \Theta_0^{(i)}
\end{equation}
and their empirical covariance operator by
\begin{equation}
    \Sigma_0^J = \frac{1}{J}
    \sum_{i=1}^J (\Theta_0^{(i)}-\widetilde{\theta}_0^J)
    \otimes (\Theta_0^{(i)}-\widetilde{\theta}_0^J).
\end{equation}
The optimal approximation error of  the empirical covariance is given by
\begin{equation} \label{eq:approximation_error1}
\| \Sigma_0^J - C_0\|_{\mathcal{L}(H_1,H_1)}  \asymp
J^{-2\alpha -1}.
\end{equation}
We can achieve this error by using a truncated singular value decomposition of $C_0$ and using this in place of the empirical approximation of the true covariance. Alternatively, the ensemble can be constructed using the Nystr\"om method, in which case (\ref{eq:approximation_error1}) holds in expectation only. See \cite{Parzer} for proof that this method has approximation error $J^{-2\alpha -1}$. If we set 
\begin{equation}
J = D(n)+1
\end{equation}
the approximation error (\ref{eq:approximation_error1}) becomes smaller than the minimax error provided (\ref{eq:oversmoothing}) is replaced by the smoothing condition
\begin{equation} \label{eq:oversmoothing2}
    \beta \le 2 \alpha + 1 
\end{equation}
on the prior Gaussian distribution. Note that (\ref{eq:oversmoothing2}) implies (\ref{eq:oversmoothing}) since $p>0$.
\begin{cor}
If $ \tilde{\theta}_0^J = \widehat{\theta}_0 $, and for $J=D(n) +1$, then
    $\| \tilde{\theta}_t^J - \widehat{\theta}_t \|_1 \lesssim J^{-2\alpha -1}$ for all $t > 0$. 
\end{cor}
\begin{proof}
For $t$ fixed, we have that 
    \begin{align}
        \| \tilde{\theta}_t^J - \widehat{\theta}_t \|_1 = \left\|\left[ \Sigma_0^J G^T \left(G  \Sigma_0^J G^T + \frac{1}{t}I\right)^{-1} Y\right] -  \left[C_0 G^T \left(G  C_0 G^T + \frac{1}{t}I\right)^{-1} Y \right]\right\|_1
    \end{align}
Then, let
\begin{align}
    K^J &=  \Sigma_0^J G^T \left(G  \Sigma_0^J G^T + \frac{1}{t}I\right)^{-1}, \\
     K &= C_0 G^\ast \left(G  C_0 G^\ast + \frac{1}{t}I\right)^{-1}.
\end{align}
We then have that 
\begin{equation}
    \| K^J - K \|_{\mathcal{L}(H_1,H_1)} \asymp J^{-2\alpha - 1}
\end{equation}
by linearity of the norm and the assumption on the approximation error \eqref{eq:approximation_error1}. It follows that
\begin{align}
      \| \tilde{\theta}_t^J - \widehat{\theta}_t \|_1 &= \|Y(K^J - K) \|_1 \nonumber \\
      & \lesssim \|Y\|_2 \|(K^J - K) \|_{\mathcal{L}(H_1,H_1)} \nonumber \\
      & \lesssim J^{-2\alpha - 1}. 
\end{align}
Plugging in $J=D(n) +1$ we get that $J^{-2\alpha - 1} = n^{-2 \alpha -1 / 2p +1} + 1 \leq n^{\beta / (2\alpha + 2p +1)}$ when \eqref{eq:oversmoothing2} is satisfied.  
\end{proof}
The mean-field EnKBF formulation (\ref{eq:DEnKBF}) is now replaced by an interacting finite particle EnKBF formulation
\begin{equation} \label{eq:DEnKBF_particle}
\frac{{\rm d}\Theta_t^{(i)}}{{\rm d}t} = n \Sigma_t^J G^T  \left\{ Y - \frac{1}{2} G \left( \Theta_t^{(i)} + \widetilde{\theta}_t^J\right) \right\} 
\end{equation}
for $i=1,\ldots,J$. 
\begin{cor}
    Denote the singular value decomposition of $C_0$ truncated at the $J^{\rm th}$ singular value by $T^J(T^J)^T = C^0$, where $T^J: \mathbb{R}^J \rightarrow H_1$ is a linear operator. By truncated, we mean that $T^J(T^J)^T = P_J(T^\infty (T^\infty)^\ast)$ where $P_J$ is defined in \eqref{eq:projector}. Let $Q_J$ be the projector of $H_1$ onto the $\mathbb{R}^J$ subspace of $H_1$, that maps the first $J$ coordinates to $\mathbb{R}^J$. Then, for all $t > 0$ we have that 
    \begin{equation}
        \|Q_J\Tilde{\theta^J}_t - \widehat \theta_t
         \| \lesssim J^{-2 \alpha -1}.
    \end{equation}
\end{cor}
\begin{proof}
     Fix $t$. Then we have that 
     \begin{equation*}
         \| Q_J \tilde \theta^J_t - \tilde \theta^J_t  \| = 0
     \end{equation*}
\end{proof}
Again plugging in $J=D(n) +1$ we get that $J^{-2\alpha - 1} = n^{-2 \alpha -1 / 2p +1} + 1 \leq n^{\beta / (2\alpha + 2p +1)}$ when \eqref{eq:oversmoothing2} is satisfied.  So, using a finite-dimensional finite ensemble size does not add error of order higher than the minimax rate. 
We can then consider the projected version of (\ref{eq:DEnKBF_particle}) which is given as 
\begin{equation} \label{eq:DEnKBF_particle2}
\frac{{\rm d}\Theta_t^{(i)}}{{\rm d}t} = n Q_J(\Sigma_t^J G^\ast)   \left\{ Q_J (Y - G \widetilde{\theta}_t^J )-
\frac{1}{2} Q_J (G \left( \Theta_t^{(i)} - \widetilde{\theta}_t^J)\right) \right\}. 
\end{equation}
The stopping criterion (\ref{eq:stopping_EnKBF1}) can be replaced by using the projected residuals 
\begin{equation} \label{eq:stopping_EnKBF2}
    t_n := \inf\,\left\{
    t > 0: \|Q_J (Y-G\widetilde{\theta}_t^J)\|^2_2 \le \kappa \right\}.
\end{equation}

It should be noted that (\ref{eq:DEnKBF_particle2}) effectively constitutes a finite-dimensional system of ordinary differential equations since 
\begin{equation}
    \Theta_t^{(j)} = \sum_{i=1}^J \Theta_0^{(i)}m^{(i,j)}_t
\end{equation}
for appropriate time-dependent coefficients $m_t^{(i,j)}$, $i,j = 1,\ldots,J$ \cite{reichcotter15}. 
%
\section{Algorithmic details} \label{sec:numerical_implementation}
%
We now give the discrete time formulation for implementing \eqref{eq:DEnKBF_particle2}, which we rewrite as 
\begin{subequations} \label{eq:DEnKBFBF}
\begin{align} 
\frac{\rm d}{{\rm d}t}\Theta_t &= -\frac{1}{2} n \Sigma^J_t G^{\rm T}  \left(G\Theta_t+Gm_t-2Y \right) .
\end{align}
With the understanding that here everything has been projected onto $\mathbb{R}^J$. 
\end{subequations}
 Denote the ensemble of $J$ parameter values at time $t_k \ge 0$ by $\theta_k^{(i)}$. The initial conditions are given by
\begin{equation} \label{eq:ICs}
\theta_0^{(i)} \sim {\rm N}(0,C_0)
\end{equation}
for $i=1,\ldots,J$. Let us denote the empirical mean of the ensemble by $m_k^{J}$  and the empirical covariance matrices by $C_k^J$. We introduce the empirical covariance matrix
between $\Theta$ and $G\Theta$, denoted by $\mathcal{C}_n^{J} \in \mathbb{R}^{D(n) \times D(n)}$, 
as well as the empirical covariance matrix of  $G\Theta$, denoted by $S_k^{J} \in \mathbb{R}^{D(n)\times D(n)}$, which is given by
\begin{equation}
\label{eq:em_cov}
    S_k^J = \frac{1}{J-1} \sum_{i=1}^J (
    G\theta_k^{(i)} -  G\tilde{\theta}^J_k)(G\tilde{\theta}^J_k- G\theta_k^{(i)})^{\rm T},
\end{equation}
where $G\tilde{\theta}^J_k$ denotes the empirical mean of $G\theta$. Similarly,
\begin{equation}
\label{eq:em_cross_cov}
    \mathcal{C}_k^{J} = \frac{1}{J-1} \sum_{i=1}^J (
    \theta_k^{(i)} - \tilde{\theta}^J_k)(G\theta_k^{(i)}) - G\tilde{\theta}^J_k)^{\rm T}.
\end{equation}

The resulting deterministic discrete-time update formulas, which follow from (\ref{eq:DEnKBFBF}), are given as
\begin{subequations} \label{eq:numerical_DEnKBF}
\begin{align}
\theta_{{k+1}}^{(i)} &= \theta_{k}^{(i)} - \frac{1}{2} K_{k} \left(G\theta_{k}^{(i)} + G\tilde\theta^J_k - 2Y
\right)
\end{align}
\end{subequations}
with the Kalman gain matrix
\begin{equation} \label{eq:Kalman_gain_numerical}
K_{k} = \Delta t \,\mathcal{C}_{k}^J \left(\Delta t S_{k}^J + I \right)^{-1}.
\end{equation}

The standard discrepancy principle stops the iteration of the EnKBF  whenever 
\begin{equation} \label{eq:numerical_dp_stopping}
   k_{\rm dp} = \inf \left\{k \geq k_0 : \|G(\tilde{\theta}^J_k) - Y\|^2 \leq \kappa \right\}.
\end{equation}
where $\kappa = C D(n) / n$ and $0<C \le 1$ is such that $| \kappa  - D(n) \rho| \leq C \sqrt{D(n)} / n$ and $k_0$ is the initial time. See Section \ref{sec:theory}.
Combining the deterministic updates, the recalibration of the initial ensemble, and the discrepancy-based stopping criterion, we have the resulting algorithm.

\newenvironment{algocolor}{%
   \setlength{\parindent}{0pt}
   \itshape
}{}

\begin{algorithm} 
      \caption{Deterministic EnKF}
      \label{algo:deter_enkf}
      \begin{algocolor}
      \begin{algorithmic}
      \Require $J > 0,\  m_0, \ \tilde  \Delta_0, \ Y,\ G, \  \rho$ 
       \State $\Theta_0 \gets \texttt{initialize}(J, m_0, \tilde \Delta_0)$ \Comment{$\Theta  \in \mathbb{R}^{D(n) \times J}$}
        \State $R_0 \gets ||G( \tilde \theta^J_0) - Y||^2 $ 
        \State $\kappa_{\rm dp} = D(n) \delta^2$ \Comment{see \eqref{eq:kappa}}
       \While{$R_k < \kappa_{\rm dp}$}  
            \State $K_{k} \gets \Delta t \,\mathcal{C}_{k}^J \left(\Delta t S_{k}^J + I \right)^{-1}$ \Comment{$\mathcal{C}_{k}^J$ see \eqref{eq:em_cross_cov},  $\Sigma_{k}^J$ see \eqref{eq:em_cov}}
           \For {$i \in \{1,..,J\}$} 
           \begin{align*}
               \theta_{{k+1}}^{(i)} &\gets \theta_{k}^{(i)} - \frac{1}{2} K_{k} \left(G\theta_{k}^{(i)} + {G}m^J_k - 2Y
            \right) 
           \end{align*}
           \EndFor
          \State $R_{k+1} \gets ||G\left(\Tilde{\theta}^J_{k+1}\right) - Y||^2 $ 
       \EndWhile
       \State \textbf{Return} $\Theta_k$
    \end{algorithmic}   
  \end{algocolor}
\end{algorithm}

\section{Numerical Examples}
\label{sec:numerics}
In this section, we will numerically demonstrate the theoretical results of \cref{sec:theory}. We will then extend the theoretical results to a numerical example of a non-linear problem, where the EnKF is now a Monte Carlo approximation of the posterior. In the non-linear setting, the EnKF provides a Monte Carlo Gaussian approximation to the true posterior. In \cite{nickl20}, it was shown that a Bernstein von Mises theorem held for the Schrödinger equation, which states asymptotically that the posterior is approximately Gaussian. Moreover, in \cite{sanz2023inverse} it was shown that if sufficient particles are used, it may be possible to well approximate this posterior, which is close to a Gaussian. The code used to generate these results can be found at  \url{https://github.com/Tienstra/EarlyStoppingBIP}

\subsection{Linear Examples}
\label{subsec:linear}
In this section, we demonstrate the performance of the discrepancy principle-based stopping rule. We choose the forward operator to have $p=1/2$.  We choose $\theta^\dagger \in S^\beta := \{\theta \in H_1 : \sum^\infty_{i=1}  \theta_i^2 i^{2\beta} < \infty \}$,\hspace{0.3cm} $\beta > 0$. We initialize the ensemble with entry wise prior $\theta_i \overset{iid}{\sim} \mathcal{N}(0, i^{-1-2\alpha})$. We generate noisy observations by fixing the dimension of $\theta^\dagger$ to $100$ and draw independent standard Gaussian noise with differing noise levels. We set $\kappa_{\rm dp} = C D(n) /n$. For each iteration, we compute the residual and record the first iteration such that $R_n^2 \leq \kappa_{\rm dp}$. We consider two different, $\theta^\dagger$ with varying smoothness levels. We show the effects as the noise level decreases, which in expectation is analogous to when the sample size grows. Our results can be found in \cref{fig:linear}.
\begin{figure}[!hbt]  
    \centering
    \includegraphics[width=0.95\linewidth]{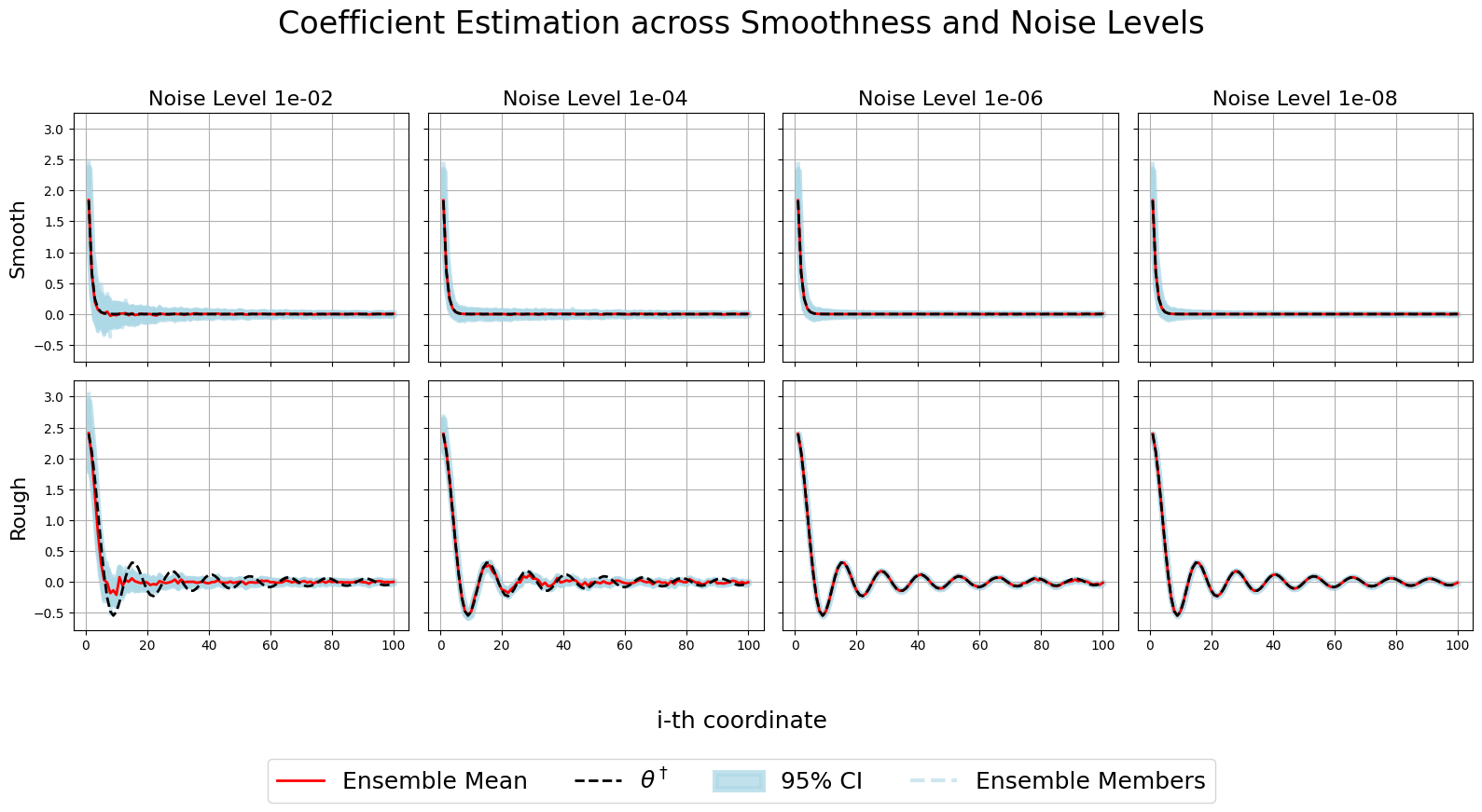}
   \caption{Here we plot the results of running \cref{algo:deter_enkf} for two different smoothness levels. We stop when $R_n^2 \leq \kappa_{\rm dp} = C  D(n) /n$, with $C =1$ for all experiments. The top row has ground truth coefficients $\theta_{{\rm rough},i} = 5 \sin (0.5i) i^{-1}$ and the bottom row as  ground truth coefficients $ \theta_{{\rm smooth},i} = 5 \text{exp}(-i)$ for $i=1,...,100$. From left to right, we decrease the noise level of the observations, thus simulating an increase in sample size. This is valid as in the linear case, the $ith$ observations are i.i.d. The prior is the same for each experiment with $\alpha =1$, and the forward operator is fixed with $p=1/2$. We see that we can adapt to both functions and that, as the noise level decreases, there is a slight contraction of the ensemble members. The $95\%$ credible region is the shaded region and is estimated by the $95\%$ quantiles of the ensemble members. 
     \label{fig:linear}}
\end{figure}
where $\theta_{{\rm rough},i} = 5 \sin (0.5i) i^{-1}$ and $ \theta_{{\rm smooth},i} = 5 \text{exp}(-i)$ for $i=1,...,100$. We show our results using the deterministic algorithm. We chose the ensemble size to be $100$. We note that a smaller ensemble size can work in practice if one is interested in the mean alone. To run the algorithm, one must choose $\alpha$ and, additionally, choose $\kappa$, which depends on estimating the noise level. From Section \ref{sec:theory}, one should choose a relatively smooth prior so that it can adapt to the true smoothness. Additionally, one should choose a small $C$ to mitigate against stopping too early, as the lower bound on $\tau_{\rm dp}$ should be satisfied to have the appropriate posterior variance.

\subsection{Non-Linear Example} \label{sec:non-linear}
%

In this section, we consider non-linear forward maps $\mathcal{G}:H \rightarrow \mathbb{R}^d$. Common examples in the literature are one-dimensional PDEs such as the Darcy flow or the Schrödinger equation, see \cite{NicklWang}, and  \cite{Koers2024}. The true parameter will be assumed to be in $\mathbb{R}^d$. We will have $N$ noisy observations, which are solutions of a PDE approximated on some $ D$-dimensional grid. As $\mathcal{G}$ is now non-linear, the posterior is no longer Gaussian. In \cite{Koers2024}, they propose to first linearise the problem, then do Bayesian inference and then pull back the results to the original problem.  We can apply our results to the linear problem, so long as the new forward operator has polynomially decaying eigenvalues and the covariance operator is chosen correctly. Here instead, we rely on (Monte Carlo) ensemble Kalman-Bucy filter implementations for nonlinear $\mathcal{G}$, in that the EnKF evolves the initial ensemble of particles from the prior into a Gaussian approximation of the true posterior, where one hopes that as the number of particles goes to $\infty$, and the noise level goes to zero the samples are a good approximation of samples from the true posterior. The EnKF and variants thereof have been successfully applied to non-linear problems, see \cite{blömker2021continuoustimelimitstochastic}, and \cite{Parzer}. We are then interested in applying the early stopping to the evolution of approximate distributions. We mention that early stopping based on the discrepancy principle for gradient descent in the non-linear inverse problem setting has been studied in \cite{Hanke1995}. In this paper, they considered a localisation around the ground truth parameter of the non-convex minimisation problem. It would be of interest to see if such a technique could be used to extend the results of this paper. However, that is beyond the scope of this paper. In the next sections, we give the details of the one-dimensional Schrödinger inverse problem and show how early stopping, as considered in \cref{sec:theory}, works numerically. 

\subsection{One dimensional Schrödinger Equation} \label{subsec:schrodinger_detials}
In this section, we repeat the details of the Bayesian inverse problem arising from the  Schrödinger equation, which can be found in detail in \cite{NicklWang}. Let $\mathcal{O}$ be a bounded subset of $\mathbb{R}^d$, for $d \in \mathbb{N}$, and $\Theta$ some parameter space. Consider a family of $\{\mathcal{G}(\theta): \theta \in \Theta \}$ of real-valued valued bounded ''regression" functions, where $L^2 (\mathcal{O})$ defines the square integral functions with respect to the Lebesgue measure. Then let 
\begin{equation*}
    \mathcal{G} : \Theta \rightarrow L^2(\mathcal{O}).
\end{equation*}
The observations are modelled as
\begin{equation}
\label{model}
    Y_i = \mathcal{G}(\theta)(X_i) + \epsilon_i, \quad i=1,..,N
\end{equation}
where $\epsilon_i \sim \mathcal{N}(0,1)$, $X_i \sim Uniform(\mathcal{O})$. 
Suppose that $\mathcal{G}(\theta)=u_{f_\theta}$ arises as the unique solution to $u=u_{f_\theta}$ in the following elliptic boundary value problem for the Schrödinger equation 
    \begin{equation*}
        \begin{cases}
        \frac{1}{2} \Delta u - f_\theta u =g \quad \text{on } \mathcal{O} \\
        u = h \quad \text{on } \partial \mathcal{O}. 
    \end{cases}
    \end{equation*}  
    As the observations are observed on some $D$ finite-dimensional grid. The log-likelihood can be written as
    \begin{equation}
    \label{likelihood}
        \ell_N(\theta) = \frac{1}{2} \sum^{N}_{i=1} [ Y_i -\mathcal{G}(\theta)(X_i) ]^2, \quad \theta \in \mathbb{R}^D.
    \end{equation}
    In \cite{NicklWang}, they construct a Gaussian prior where the covariance is built from the eigenvalues of the Laplacian. Denote the eigenvalues of the Laplacian by $(\lambda_k)_{k \in \mathbb{N}}$, and recall that $D$ corresponds to the dimension of the grid. Fix $\alpha$. Then the prior is given as 
        \begin{equation}
             \label{eq:prior_nonlinear}
            \theta \sim \Pi = \Pi_N  \sim \mathcal{N}(0,  N^{-d/(2\alpha +d)}\Lambda_\alpha ^{-1})
        \end{equation}
        where $\Lambda_\alpha = diag(\lambda^\alpha_1, ..., \lambda^\alpha_D)$.
        The posterior given observations $Z^{(N)}= ((Y_1, X_1),...,(Y_N,X_N))$ is
    \begin{align}\label{eq:schrodinger_post}
        \pi(\theta \mid Z^{(N)}) & \ \propto \ e^{\ell_N(\theta) \pi(\theta)} \nonumber \\
        & \ \propto \text{ exp}\left\{-\frac{1}{2} \sum_{i=1}^N (Y_i - \mathcal{G}(\theta)(X_i))^2 - \frac{N^{d/(2 \alpha +d)}}{2} ||\theta||_{\alpha}^2\right\}.   
    \end{align}
    The point estimator is then
    \begin{equation}
        \widehat{\theta}_{MAP} \in \underset{\theta \in \mathbb{R}^D}{\text{arg max}} \pi(\theta \mid Z^{(N)})
        \end{equation}

\subsection{Numerical Implementation}
 \label{subsec:nonlinear_numerics}
In this section, we will give the numerical details of implementing the problem described \cref{subsec:schrodinger_detials}. Rather than run our numerics in sequence space, which is a tool for the theoretical analysis, here in the implementation, we stay in function space. We need to first solve the forward problem to generate observations. To achieve this, will observe $Z^{(N)}$ over a finite grid of size D. We choose the grid by setting 
\begin{align*}
    x_i = \frac{2\pi i}{D+1}, \hspace{1cm} i=0,...,D
\end{align*}
and also require that $u_0 = u_D$, $g_0 = g_D$. Then we have 
\begin{align}
    \mathcal{L}^{D}_f \cdot u = g
\end{align}
where $\mathcal{L}^{D}_f = \frac{1}{h^2} \frac{\Delta_D}{2} - I_{f}$ is the finite approximation of the $\mathcal{L}_f$, for fixed $f$, and here $f = (f(x_i),..., f(x_D))$ and   $u = (u(x_i),..., u(x_D))$, and $I_f = \text{diag} (f)$ and $h=\frac{2 \pi}{D}$. The solution is then given by 
\begin{align}
    u_f = h^2 [\frac{\Delta_D} {2} - I_f ]^{-1} g.
\end{align}
The observations are given by 
\begin{align} \label{eq:nonlinear_model}
    y &= u_f + \Xi \nonumber \\
    & = h^2 [\frac{\Delta_D} {2} - I_f ]^{-1} g + \Xi \nonumber \\
    & := \mathcal{G}_f^{D} g + \Xi
\end{align}
where $f$ is the now the parameter of the forward map $\mathcal{G}_f^{D}$ with $D$ evaluated on the grid points. We parameterize $f_\theta = e^\theta$ so that $f > 0$. 
\begin{align}
     [\frac{\Delta_D} {2h^2} - I_{f_\theta} ] = \begin{bmatrix}
     -1/h^2-f_0 & 1 & 0 & ... & 1/2h^2 \\
     1/2h^2 & -1/h^2 - f_1 & 1 & ... & 0 \\ 
     \vdots \\
     1/2h^2 & 1/2h^2 & 0 & ... & -1/h^2-f_D
     \end{bmatrix}
\end{align}
where $f_i = f_\theta(i) = f(\theta(x_i))$. 

To simulate observations, we choose a $D$ and a source condition
\begin{align}
    g_i = exp\left(-\frac{(x_i-L/2)^2}{10}\right) - \text{mean}(g)
\end{align}
where $L=2 \pi$, the length of the interval. We choose the initial ensemble (see \cite{Garbuno-Inigo2020}) to be normally distributed with mean $0$ with covariance matrix $P_0$ where 
\begin{equation}
\label{eq:precision_matrix}
    P_0^{-1} = 4h\left(\frac{\mu}{D+1}I_{D+1}I_{D+1}^T-\Delta_H\right)^2
\end{equation}
where $\Delta_H$ is the second order finite difference operator, and $\mu$ is chosen so that the mean of $\{u_i\}_{i=0}^{D}$ is close to zero. This is a regularised version of \cref{eq:prior_nonlinear} in function space.

The ground truth is $f_{\theta^\dagger}=e^{\theta^\dagger}=exp(0.5\sin(x_i))$ evaluated over $101$ grid points from the interval $[0,2\pi]$. We used the stopping condition $R_n^2 \leq CD(n) /n$ where $C = 0.5$ across all experiments. We initialised the ensemble of size 50 such that 
\begin{equation}
\label{eq:nonlinear_prior}
    \Theta_0 \sim \exp(\mathcal{N}(0, P_0))
\end{equation}
The results are shown in \ref{fig:nonlinear}.
\begin{figure}[hbt!]
    \centering
    \includegraphics[scale=0.4]{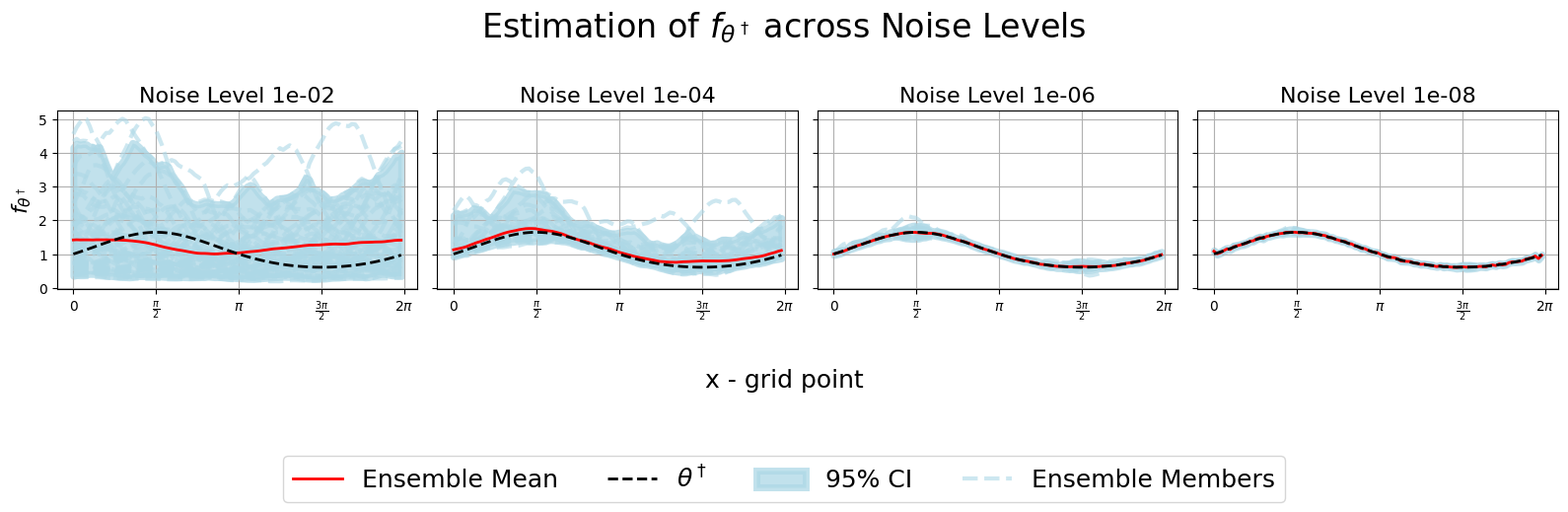}
   \caption{Here we plot the result of running \cref{algo:deter_enkf} on the direct nonlinear problem with the prior specified in \cref{eq:nonlinear_prior}. From left to right, we see the noise level decrease and the ensemble particles collapse. The shaded region is the $95\%$ credible region as estimated by the $95\%$ quantiles of the ensemble members. We also see that the mean goes to the ground truth.}
    \label{fig:nonlinear}
\end{figure}
As the EnKF is only an approximation of the true posterior, we also investigated how well the EnKF worked compared to HMC sampling of the posterior \cref{eq:schrodinger_post}. We used the estimated $\tau_{\rm dp}$ from this experiment as the scaling term in the prior. We then consider 
\begin{equation}
      \pi(\theta \mid Z^{(N)})  \ \propto \ e^{\ell_N(\theta) \pi(\theta) }
\end{equation}
where 
\begin{equation}
    \pi(\theta) \sim \mathcal{N}(0, \tau_{\rm dp}^2 P_0)
\end{equation}
to be a ground truth posterior. The results can be found in \cref{fig:hmc_vs_enkbf}.

\begin{figure}[hbt!]
    \centering
    \includegraphics[width=0.9\linewidth]{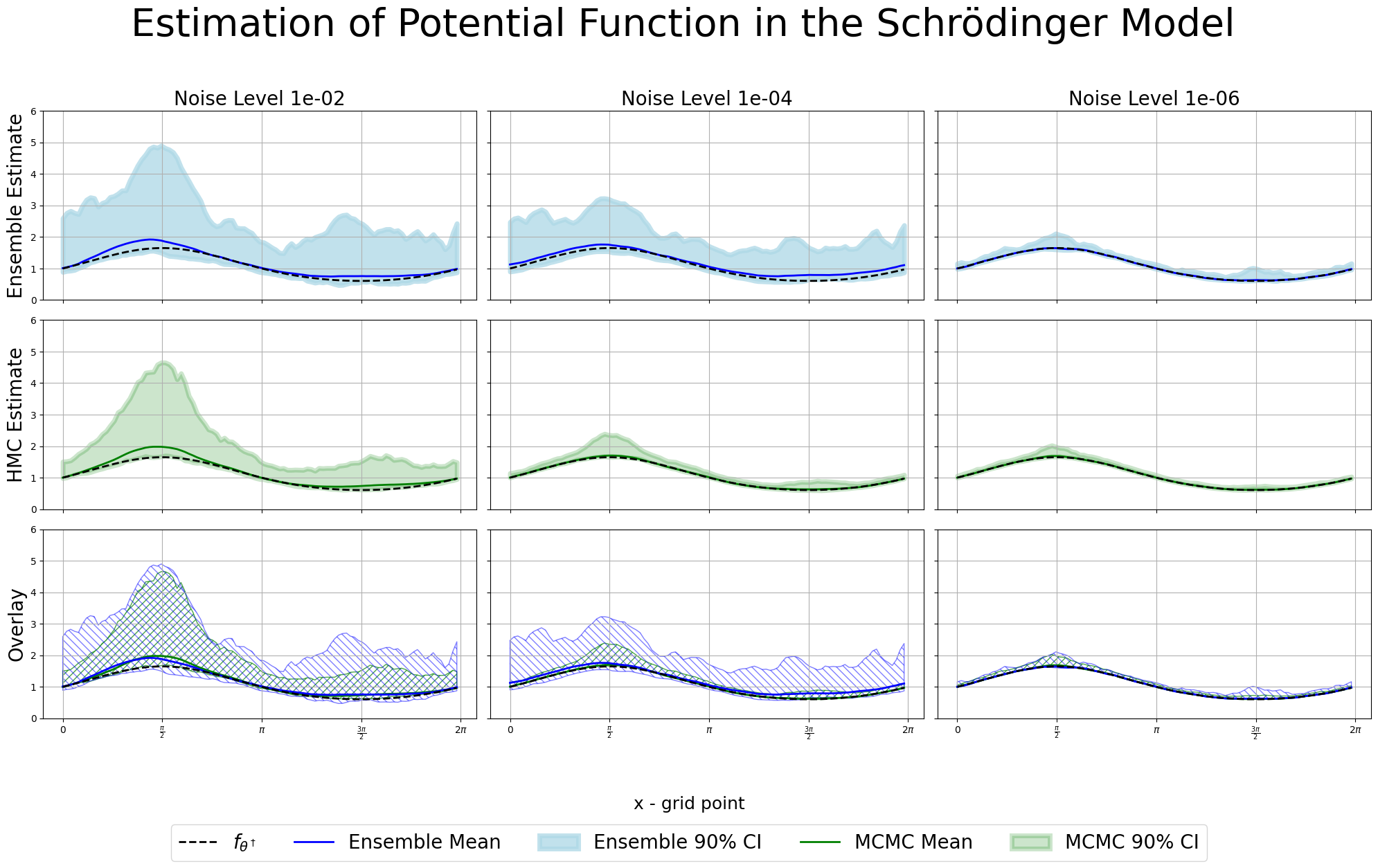}
   \caption{In the top row, we have plotted the results of running \cref{algo:deter_enkf}, with stopping criterion $R \leq C D(n)/n$ where $C=0.1$ when the noise level is $1e-2$, and $C=0.5$ else. The middle row is the resulting HMC estimates. The bottom row is a comparison of the particles as samples from the posterior and the samples from the HMC overlayed. We see that the EnKF overestimates the variance of the posterior, but still maintains the correct shape. For all noise levels, the r-hat value was 1. }
    \label{fig:hmc_vs_enkbf}
\end{figure}

%
\section{Conclusion} \label{sec:conclusion}
%

Prior selection is an important decision to make in Bayesian methods. We have seen that the scale parameter determines the amount of regularisation the prior introduces into the problem. We have also seen in the analysis in Section \ref{sec:theory} that the scale parameter strongly influences the posterior variance. Therefore, in order to derive a posterior that has the desired properties, the correct choice of scale parameter is critical. We have considered an adaptive empirical method for choosing the scale parameter of the covariance of a Gaussian prior using early stopping. It was shown in the linear case that the link between the choice of how much regularisation is needed and the choice of the prior is, in effect, the choice of how much to scale the covariance operator. The stopping rule was dependent on the dimension of the discretisation, which is an upper bound on the effective dimension of the observations. In addition, the asymptotic behaviour of the stopped Bayesian estimator and the associated stopped posterior were analysed. It was explicitly shown that the stopped estimator is minimax optimal for $\widetilde{\theta}$ and that the posterior contracts at the same rate. We also showed that this same rate transfers over to $\theta$. We also showed that the stopped posterior's frequentist coverage tends to 1. We showed that this method is adaptive if the prior smoothness $\alpha$ is chosen to be such that  $\beta < 1 + 2p + 2 \alpha$ holds. Finally, we tested our method numerically using the EnKF. This algorithm allowed us to reformulate our problem in terms of an iterative process of evolving Gaussian distributions. Through this, we were able to derive a truly iterative process that needed to be stopped, with a mild condition on the initial distribution. 

The results of this paper depend on the structural assumptions made in  \cref{sec:theory}. In particular, we point out that the proofs depend on the linearity of the forward operator, the assumption on the filter function, and the structure of the covariance operator. Moreover, the ensemble Kalman filter introduced in  \cref{sec:EKI} is an exact method only in the linear setting, but ensemble Kalman methods have achieved success in cases where the forward operator is non-linear and a Gaussian approximation of the posterior is appropriate. Possible paths, such as linearisation, localisation, and Gaussian approximation, as methods to extend our results were mentioned in \cref{sec:non-linear}. We saw in \cref{fig:hmc_vs_enkbf} that the stopped EnKF provides a good approximation for the true posterior in the non-linear setting. It was also noted that first linearising the problem, as in \cite{Koers2024} provides a more direct approach to being able to apply the results of this paper, and is currently under investigation. It would also be interesting to see if the results, of the type in this paper and in \cite{BlanchardHoffmannReiss} could be proved for locally convex problems and thereby relax the linearity condition. 
%
\smallskip
\smallskip
\smallskip

\paragraph{Acknowledgements.}
This work has been funded by Deutsche Forschungsgemeinschaft (DFG) - Project-ID 318763901 - SFB1294. MT thanks Jakob Walter for his continuous revisions and feedback. MT also thanks Bernhard Stankewitz for his feedback on an early draft of this paper. 

\newpage

\bibliographystyle{abbrvnat}
%
\bibliography{refs_EarlyStoppingforEKI}
%

\appendix
\section{Appendix}
\label{sec:appendix}
In this section, we state the original form of the Lemmas and Theorems used in our paper and cite their source, where the proofs can be found. 
\begin{assump}[Assumption \textbf{R} in \cite{BlanchardHoffmannReiss}]
    \label{ass:R} 
    \hspace{0.5em}
    \begin{enumerate}
        \item[\textbf{R1.}] The function \( g(t, \lambda) \) is nondecreasing in \( t \) and \( \lambda \), continuous in \( t \) with \( g(0, \lambda) = 0 \) and \( \lim_{t \to \infty} g(t, \lambda) = 1 \) for any fixed \( \lambda > 0 \).
        \item[\textbf{R2.}] For all \( t \geq t' \geq t_\circ \), the function \( \lambda \mapsto \frac{1 - g(t', \lambda)}{1 - g(t, \lambda)} \) is nondecreasing.
        \item[\textbf{R3.}] There exist positive constants \( \rho \), \( \beta_- \), \( \beta_+ \) such that for all \( t \geq t_\circ \) and \( \lambda \in (0, 1] \), we have
        \begin{equation}
            \beta_- \min \left( (t \lambda)^\rho, 1 \right) \leq g(t, \lambda) \leq \min \left( \beta_+ (t \lambda)^\rho, 1 \right).
        \end{equation}
    \end{enumerate}
\end{assump}

\begin{assump}[Assumption \textbf{S} in \cite{BlanchardHoffmannReiss}]
    \label{ass:S}
     There exist constants \( \nu_- \), \( \nu_+ > 0 \) and \( L \in \mathbb{N} \) such that for all \( L \leq k \leq D \),
    \begin{equation}
        0 < L^{-1/\nu_-} \leq \frac{\lambda_k}{\lambda_{\lfloor k / L \rfloor}} \leq L^{-1/\nu_+} < 1.
    \end{equation}
    
\end{assump}
\begin{theorem}[Theorem 3.5 in \cite{BlanchardHoffmannReiss}]
\label{thrm:BlanHoffReiss}
Suppose Assumptions \textbf{R}, \textbf{S} hold with 
$$ \rho > \max(\nu_+, 1 + \frac{\nu_+}{2}) $$ 
and (3.3) holds for $\kappa$ with $C_\kappa \in [0, C_\circ)$. Then for all $\mu$ with $t_\tau(\mu) \leq t_\delta(\mu)$, we have
$$
\mathbb{E} \left[ \| \hat{\mu}^{(\tau)} - \mu \|^2 \right] \leq C_{\tau, \delta} \mathbb{E} \left[ \| \hat{\mu}^{(t_\delta)} - \mu \|^2 \right].
$$

For all $\mu$ with $t_\tau(\mu) \geq t_\delta(\mu)$, we obtain
$$
\mathbb{E} \left[ \| \hat{\mu}^{(\tau)} - \mu \|^2 \right] \leq C_{\tau, \tau}(t_\tau \wedge \lambda_D^{-1})^2 \mathbb{E} \left[ \| \hat{\mu}^{(t_\tau)} - \mu \|^2_A \right].
$$

The constants $C_{\tau, \delta}$ and $C_{\tau, \tau}$ depend only on $\rho$, $\beta_-$, $\beta_+$, $L$, $\nu_-$, $\nu_+$, $C_\circ$, $C_\kappa$.
    
\end{theorem}

\begin{lemma}{Lemma 8.2 in \cite{knapik2011}}
    \label{lem:8.2}
    For any $t, v \geq 0, u > 0$, and $(\xi_i)$ such that $|\xi_i| = i^{-q - 1/2} \mathcal{S}(i)$ for $q > -t/2$ and a slowly varying function $\mathcal{S} : (0, \infty) \to (0, \infty)$, as $N \to \infty$,
    \[
    \sum_i \frac{\xi_i^2 i^{-t}}{(1 + N i^{-u})^v} \sim
    \begin{cases} 
    N^{-(t + 2q)/u} \mathcal{S}^2(N^{1/u}), & \text{if } (t + 2q)/u < v, \\
    N^{-v} \sum_{i \leq N^{1/u}} \frac{\mathcal{S}^2(i)}{i}, & \text{if } (t + 2q)/u = v, \\
    N^{-v}, & \text{if } (t + 2q)/u > v.
    \end{cases}
    \]
    
    Moreover, for every $c > 0$, the sum on the left is asymptotically equivalent to the same sum restricted to the terms $i \leq cN^{1/u}$ if and only if $(t + 2q)/u \geq v$.
\end{lemma}

\begin{theorem}[Theorem 4.2 in \cite{knapik2011}]
\label{thrm:coverage}
    Let $\mu_0$, $(\lambda_i)$, $(\kappa_i)$, and $\tau_n$ be as in Assumptions \cref{subsec:structural_ass}, and set 
$$ \beta = \tilde{\beta} \wedge (1 + 2\alpha + 2p). $$ 
The asymptotic coverage of the credible region (4.2) is:
\begin{enumerate}
    \item $1$, uniformly in $\mu_0$ with $\|\mu_0\|_\beta \leq 1$, if 
    $$ \tau_n \gg n^{(\alpha - \tilde{\beta}) / (1 + 2\tilde{\beta} + 2p)}; $$
    in this case $\tilde{r}_{n, \gamma} \asymp r_{n, \gamma}$.

    \item $1$, for every fixed $\mu_0 \in S^\beta$, if $\beta < 1 + 2\alpha + 2p$ and 
    $$ \tau_n \asymp n^{(\alpha - \tilde{\beta}) / (1 + 2\tilde{\beta} + 2p)}; $$
    along some $\mu_0^n$ with 
    $$ \sup_n \|\mu_0^n\|_\beta < \infty, $$
    if 
    $$ \tau_n \asymp n^{(\alpha - \tilde{\beta}) / (1 + 2\tilde{\beta} + 2p)} \quad (\text{any } c \in [0, 1)). $$

    \item $0$, along some $\mu_0^n$ with 
    $$ \sup_n \|\mu_0^n\|_\beta < \infty, $$
    if 
    $$ \tau_n \ll n^{(\alpha - \tilde{\beta}) / (1 + 2\tilde{\beta} + 2p)}. $$
\end{enumerate}
\end{theorem}

\begin{remark}
\label{rem:caclulation_thrm4.1}
    Here we compute the rate of contraction in Theorem 4.1 of \cite{knapik2011} under the condition that $\beta \leq 2\alpha +2p + 1$ and $\tau_n^2 = n\tau_n^2$. The expected error (unsquared) is 
    \begin{equation}
       \epsilon_n =  (n\tau_n^2)^{-  \beta/(1 + 2\alpha  + 2p)} + \tau_n(n\tau_n^2)^{-\alpha/(1+ 2\alpha + 2p)}
    \end{equation}
    and the optimal $\tau_n^* = n^{(\alpha -\beta)/(2 \beta + 2p + 1)}$. 
    We see that the bias estimate is the same as the one in \cref{eq:bias_bound} as 
    \begin{equation}
        \text{bias}^2 = (n\tau_n^2)^{-  2\beta/(1 + 2\alpha  + 2p)} = (n\tau_n^2)^{-  \beta/(1/2 + \alpha  + p)}
    \end{equation}
    We then compute that 
    \begin{align}
        \epsilon_{\tau_n*}^2 = n^{\frac{\alpha - \beta}{2b + 2p + 1}}.
\left(
    \left(
        n^{\frac{2(\alpha - \beta)}{2b + 2p + 1} + 1}
    \right)^{-\alpha / (2\alpha + 2p + 1)}
    +
    \left(
        n^{\frac{2(\alpha - \beta)}{2b + 2p + 1} + 1}
    \right)^{-\beta / (2\alpha + 2p + 1)}
\right)
    \end{align} 
\end{remark}

\end{document}